\documentclass[11pt,UTF-8,reqno]{amsart}
\usepackage{enumerate}
\usepackage{mhequ}
\usepackage[margin=1.2in]{geometry}
\usepackage{amssymb,url,color, booktabs,nccmath}
\usepackage{mathrsfs}
\usepackage{enumitem}
\usepackage{graphicx}
\usepackage{tikz}
\usepackage{graphicx}

\usepackage{color}
\usepackage[colorlinks=true]{hyperref}
\hypersetup{
    linkcolor=blue,          
    citecolor=red,        
    filecolor=blue,      
    urlcolor=cyan
}

\definecolor{darkergreen}{rgb}{0.0, 0.5, 0.0}

\numberwithin{equation}{section}
\def\theequation{\arabic{section}.\arabic{equation}}
\newcommand{\be}{\begin{eqnarray}}
\newcommand{\ee}{\end{eqnarray}}
\newcommand{\ce}{\begin{eqnarray*}}
\newcommand{\de}{\end{eqnarray*}}
\newtheorem{theorem}{Theorem}[section]
\newtheorem{lemma}[theorem]{Lemma}
\newtheorem{remark}[theorem]{Remark}
\newtheorem{definition}[theorem]{Definition}
\newtheorem{proposition}[theorem]{Proposition}
\newtheorem{Examples}[theorem]{Example}
\newtheorem{corollary}[theorem]{Corollary}
\newtheorem{assumption}{Assumption}[section]

\def\${|\!|\!|}

\def\proj{\mathbf{p}}

\def\Re{{\mathrm{Re}}}

\def\var{{\mathrm{Var}}}
\def\cov{{\mathrm{Cov}}}

\def\eps{\varepsilon}

\def\p{\partial}

\def\[{{\Big[}}
\def\]{{\Big]}}
\def\<{{\langle}}
\def\>{{\rangle}}
\def\({{\Big(}}
\def\){{\Big)}}

\def\bx{{\mathbf{x}}}
\def\tr{\mathrm {tr}}

\def\Ric{{\rm Ricci}}

\def\dif{{\mathord{{\rm d}}}}

\def\Hess{{\mathord{{\rm Hess}}}}
\def\min{{\mathord{{\rm min}}}}

\def\no{\nonumber}
\def\={&\!\!=\!\!&}

 \newcommand{\eqdef}{\stackrel{\mbox{\tiny def}}{=}}

\def\cE{{\mathcal E}}
\def\cF{{\mathcal F}}

\def\cL{{\mathcal L}}

\def\cP{{\mathcal P}}
\def\cQ{{\mathcal Q}}

\def\cS{{\mathcal S}}

\def\mN{{\mathbb N}}

\def\mP{{\mathbb P}}

\def\mR{{\mathbb R}}

\def\mT{{\mathbb T}}

\def\mZ{{\mathbb Z}}

\def\1{{\mathbf{1}}}

\def\sC{{\mathscr C}}

\def\sF{{\mathscr F}}

\def\sO{{\mathscr O}}
\def\sP{{\mathscr P}}

\def\sR{{\mathscr R}}

\def\E{\mathbf E}

\def\geq{\geqslant}
\def\leq{\leqslant}
\def\ge{\geqslant}
\def\le{\leqslant}

\let\emptyset\varnothing

\def\E{\mathbf{E}}
\def\CP{\mathcal{P}}
\def\Tr{\mathrm{Tr}}
\def\q{\mathfrak{q}}

\def\div{\mathord{{\rm div}}}

\def\Re{{\mathrm{Re}}}

\def\eps{\varepsilon}

\def\p{\partial}

\def\[{{\Big[}}
\def\]{{\Big]}}
\def\<{{\langle}}
\def\>{{\rangle}}
\def\({{\Big(}}
\def\){{\Big)}}

\def\bx{{\mathbf{x}}}
\def\tr{\mathrm {tr}}

\def\Ric{{\rm Ric}}

\def\dif{{\mathord{{\rm d}}}}

\def\Hess{{\mathord{{\rm Hess}}}}
\def\min{{\mathord{{\rm min}}}}

\def\tr{{\rm Tr}}
\def\no{\nonumber}
\def\={&\!\!=\!\!&}
\def\bt{\begin{theorem}}
\def\et{\end{theorem}}
\def\bl{\begin{lemma}}
\def\el{\end{lemma}}
\def\br{\begin{remark}}
\def\er{\end{remark}}
\def\bx{\begin{Examples}}
\def\ex{\end{Examples}}
\def\bd{\begin{definition}}
\def\ed{\end{definition}}
\def\bp{\begin{proposition}}
\def\ep{\end{proposition}}
\def\bc{\begin{corollary}}
\def\ec{\end{corollary}}

\def\so{\mathfrak{so}}
\def\su{\mathfrak{su}}

\def\mfg{\mathfrak{g}}

\def\geq{\geqslant}
\def\leq{\leqslant}
\def\ge{\geqslant}
\def\le{\leqslant}

\def\div{\mathord{{\rm div}}}

\def\YM{\textnormal{\tiny \textsc{ym}}}
\def\muYM{\mu^{\YM}_{N,\beta}}
\def\dg{d(\mfg)}

\def\R{\mathbb R}
\def\C{\mathbb C}

\def\<{\langle} \def\>{\rangle}

\allowdisplaybreaks

\begin{document}

\title{A stochastic analysis approach to  lattice Yang--Mills  at strong coupling}
\author{Hao Shen}
\address[H. Shen]{Department of Mathematics, University of Wisconsin - Madison, USA}
\email{pkushenhao@gmail.com}
%
%
\author{Rongchan Zhu}
\address[R. Zhu]{Department of Mathematics, Beijing Institute of Technology, Beijing 100081, China 
}
\email{zhurongchan@126.com}

\author{Xiangchan Zhu}
\address[X. Zhu]{ Academy of Mathematics and Systems Science,
Chinese Academy of Sciences, Beijing 100190, China
}
\email{zhuxiangchan@126.com}


\date{\today}


\maketitle

\begin{abstract}
We develop a new stochastic analysis approach to 
 the lattice Yang--Mills model at strong coupling in any dimension $d>1$,
 with t' Hooft scaling $\beta N$ for the inverse coupling strength.
We study their Langevin dynamics, ergodicity, functional inequalities, large $N$ limits, and mass gap.

Assuming  $|\beta| < \frac{N-2}{32(d-1)N}$
for the structure group  $SO(N)$, or  $|\beta| < \frac{1}{16(d-1)}$ for $SU(N)$,
we prove the following results.
The invariant measure for the corresponding Langevin dynamic
is unique on the entire lattice, and the dynamic is exponentially ergodic under a Wasserstein distance.
The finite volume Yang--Mills measures converge to this unique invariant measure in the infinite volume limit, for which
Log-Sobolev and Poincar\'e inequalities hold.
These functional inequalities imply that the suitably rescaled Wilson loops for the infinite volume measure
has factorized correlations and
converges in probability to deterministic limits in the large $N$ limit,
and correlations of a large class of observables
decay exponentially, namely the  infinite volume measure  has a strictly positive mass gap. Our method improves earlier results or simplifies the proofs, and provides some new perspectives to the study of lattice Yang--Mills model. 
\end{abstract}

\setcounter{tocdepth}{2}
\tableofcontents

\section{Introduction}

The purpose of this paper is to apply stochastic analysis  and ergodic theory for Markov processes
to study the lattice Yang--Mills model with structure group $G\in \{SO(N), SU(N) \}$. 
In particular, we will consider the Langevin dynamics of these models,
and under explicit strong coupling  assumptions,
we will prove uniqueness of invariant measures in  infinite volume,
log-Sobolev and Poincar\'e inequalities, with some application in
large $N$ limits of Wilson loops and exponential decay of correlations.

Lattice discretizations of the Yang--Mills theories were first 
 proposed in the physics literature  by Wilson  \cite{Wilson1974}
which lead to well-defined Gibbs measures on collections of matrices.
We refer to \cite{Chatterjee18} for a nice review on the Yang--Mills model
and its gauge invariant discretization as well as the fundamental questions for the model. Among the literature we only mention that 
approximate computations of the Wilson loop expectations  as the size $N$ of the structure group becomes large was first suggested by 't Hooft \cite{tHooft1974}, where the Yang--Mills Hamiltonian is multiplied by $\beta N$ (known as the 't Hooft scaling), which is closely related to our present article.

The problems we discuss in this paper have been of interest and studied for decades in mathematical physics. 
A closely related
 earlier paper is by Osterwalder--Seiler \cite{OS1978},
which showed that for the lattice Yang--Mills theory, when the coupling is sufficiently strong,
the cluster expansion (or high-temperature expansion in statistical mechanics language)
for the expectation values of local observables (i.e. bounded functions of finitely many edge variables)
is convergent, uniformly in volume.
The proof of this convergent cluster expansion was sketched in \cite{OS1978} since it follows similarly as
\cite{GJS1973} for $P(\phi)_2$ model (and also \cite{MR389080}); in fact it is  simpler than the $P(\phi)_2$ model in \cite{GJS1973} since the fields are bounded in lattice Yang--Mills theory.
Moreover, as explained in \cite{OS1978},  the existence of a mass gap (exponential clustering) follows from  convergence of the cluster expansion, so do existence of the infinite volume limit and analyticity of  Schwinger functions  in the inverse coupling.
Uniqueness of infinite volume limit should also follow from cluster expansion, see e.g. \cite{MR990999} for the case of the $P(\phi)_2$ model.
We also refer to the book \cite{MR785937} for these expansion techniques and results. 
As for the large $N$ limits,
in the recent papers, factorization property of the Wilson loop expectations was proved in \cite[Corollary 3.2]{Cha} and \cite{Jafar} under the assumption that $\beta$ is sufficiently small.

Given the earlier work, we revisit these problems in this article for a number of reasons.
First of all,
 the earlier work  \cite{OS1978} didn't consider 't Hooft scaling, but if we translate their results into 't Hooft scaling where the Hamiltonian is multiplied by $\beta N$ then their condition amounts to requiring $\beta N$ to be small. 
However, to our best knowledge,
 under the 't Hooft scaling $\beta N$ 
uniqueness was not known for $\beta$ in a fixed small neighborhood of the origin when $N$ is arbitrarily large (see for instance the discussion after \cite[Theorem 3.1]{Cha});
this is the reason that \cite{Cha} and \cite{Jafar} formulated 
their large $N$ results on a sequence of $N$-dependent finite volumes.
One aim of this paper is to establish uniqueness of  infinite volume measures  for $\beta$ in a {\it fixed and explicit}  small neighborhood of the origin which is {\it uniform} in $N$,
which allows us to prove
 the existence of a mass gap and large $N$ limits of Wilson loops directly in infinite volume for this range of $\beta$.

Secondly, as another motivation of this paper, 
we develop new methods  based on stochastic analysis 
and give new proofs to these results.
In these methods, the curvature properties of the Lie groups
are better exploited via the verification of the Bakry--\'Emery condition.
In particular, this allows us to perform more delicate calculations and obtain more explicit smallness condition on inverse coupling.
As another novelty we study the Langevin dynamics 
(or stochastic quantization) and we prove uniqueness of the infinite volume measures by showing that the dynamic on the entire $\mZ^d$ has a unique invariant measure.
To this end we employed coupling methods for our stochastic dynamics,
which is a variant of Kendall--Cranston's coupling.
Such stochastic coupling arguments were used earlier in the stochastic analysis 
on manifolds, but
to our best knowledge this appears to be the first time 
that such coupling arguments are used in the setting of
statistical physics or lattice quantum field theory models
with manifold target spaces.
For our coupling arguments we will also need to introduce suitable weighted 
distances on the product manifolds, 
and in our calculations a subtle comparison between the weight parameter
and the curvature plays a key role in order to obtain ergodicity.

As the third motivation, it appears to us that 
some of the proofs in this paper are simpler.
For instance, the large $N$  results on Wilson loops
follow quickly from the Poincar\'e inequality,
which simply comes from the Bakry--\'Emery condition.
Our proof of exponential decay relies on 
some earlier ideas of  Guionnet--Zegarlinski \cite{GZ}
together with our explicit bounds on 
commutators between derivatives and Markov generators 
on Lie groups. This seems to be simpler than cluster expansion,
or at least provides some new perspectives.

\subsection{Lattice Yang--Mills}
\label{sec:YM}

We first recall the basic setup and definitions of the model.

Let $\Lambda_{L}=\mZ^d\cap L\mT^d$ be a finite $d$ dimensional lattice
with side length $L$ and unit lattice spacing, and we will consider various functions on it with periodic boundary conditions.
We will sometimes write $\Lambda=\Lambda_{L}$ for short.
We say that a lattice edge of $\mZ^d$ is positively oriented if the beginning point is smaller in lexographic order than the ending point.
Let $E^+$ (resp. $E^-$) be the set of positively (resp. negatively) oriented edges,
and  denote by $E_{\Lambda_{L}}^+$, $E_{\Lambda_{L}}^-$ the corresponding subsets
of edges with both beginning and ending points in ${\Lambda_{L}}$.  Define $E\eqdef E^+\cup E^-$ and let $u(e)$ and $v(e)$ denote the starting point and ending point of an edge $e\in E$, respectively.

We write  $G$ for the Lie group
$SO(N)$ or $SU(N)$ and $\mfg$ for the associated Lie algebra $\so(N)$ or $\su(N)$.
Note that we always view $G$ as a real manifold (even for $SU(N)$), and $\mfg$ as a real vector space, and we will write
$\dg = \dim_{\R} \mfg $.

To define the lattice Yang--Mills theory we need more notation,
for which we closely follow \cite{Cha} and \cite{SSZloop}.

A {\it path}  is defined to be a sequence of edges $e_1e_2\cdots e_n$ with $e_i\in E$ and $v(e_i)=u(e_{i+1})$ for $i=1,2,\cdots, n-1$. The path is called closed if $v(e_n)=u(e_1)$.
A {\it plaquette} is a closed path of length four which traces out the boundary of a square.
Also, let $\CP_{\Lambda_{L}}$ be the set of plaquettes whose vertices are all in $\Lambda_{L}$, and
$\CP^+_{\Lambda_{L}}$ be the subset of plaquettes $p=e_1e_2e_3e_4$ such that
the beginning point of $e_1$ is lexicographically the smallest among all the vertices in $p$ and the ending point of $e_1$ is the second smallest.

The lattice Yang-Mills theory (or lattice gauge theory)
on ${\Lambda_{L}}$ for the structure group $G$, with $\beta\in\R$ the inverse coupling constant, is the
  probability measure $\mu_{\Lambda_L, N, \beta}$  on the set of all collections $Q = (Q_e)_{e\in E_{\Lambda_L}^+}$ of $G$-matrices, defined as
\begin{equation}\label{measure}
\dif\mu_{\Lambda_{L}, N, \beta}(Q)
 \eqdef Z_{{\Lambda_{L}}, N,\beta}^{-1}
\exp\Big(  \cS (Q)\Big) \prod_{e\in E^+_{\Lambda_{L}}} \dif\sigma_N(Q_e)\, ,
\end{equation}
with
\begin{equ}[e:defS]
\cS (Q) \eqdef N\beta  \Re \sum_{p\in \CP^+_{\Lambda_{L}}} \Tr(Q_p),
\end{equ}
where $Z_{\Lambda_{L},  N, \beta}$ is the normalizing constant, $Q_p \eqdef Q_{e_1}Q_{e_2}Q_{e_3}Q_{e_4}$ for a plaquette $p=e_1e_2e_3e_4$, and $\sigma_N$ is the Haar measure on $G$.
Note that for $p\in \CP^+_{\Lambda_{L}}$ the edges $e_{3}$ and $e_{4}$ are negatively oriented,
so throughout the paper we define $Q_{e}\eqdef Q_{e^{-1}}^{-1}$ for $e \in E^{-}$, where $e^{-1}$ denotes the edge with orientation reversed.
Also, $\Re$ is the real part, which can be omitted when $G=SO(N)$.

\subsection{Main Results}

We will assume the following in our main results on  lattice Yang--Mills.

\begin{assumption}\label{ass1}
	Suppose that
	\begin{align*}
		K_\cS\eqdef\begin{cases}
			\displaystyle
			\frac{N+2}4-1-8N|\beta|(d-1)>0,&\qquad
		 G=SO(N)\;,
			\\
			\displaystyle \frac{N+2}2-1-8N|\beta|(d-1)>0\;,
			&\qquad
			G=SU(N)\;.
		\end{cases}	
	\end{align*}
\end{assumption}
Assumption \ref{ass1} is equivalent to the following {\it strong coupling} assumption:
\begin{align}\label{eq:a}
	|\beta|<\begin{cases}
		\displaystyle
		\frac{1}{32(d-1)}-\frac{1}{16N(d-1)},&\qquad
		G=SO(N)\;,
		\\
		\displaystyle \frac{1}{16(d-1)},&\qquad
		G=SU(N)\;.
	\end{cases}	
\end{align}

Define the (product) topological space $\cQ\eqdef G^{E^+}$,
which will serve as our infinite volume configuration space.
By Tychonoff's theorem $\cQ$ is compact.
For each $a>1$
we define the distance $\rho_{\infty,a}$ on $\cQ$  by
\begin{equ}[e:rho-inf]
	\rho_{\infty,a}^2(Q,Q')\eqdef \sum_{e\in E^+}\frac1{a^{|e|}}\rho^2(Q_e,Q_e'),
\end{equ}
with $|e|$ being  the distance from $0$  to $e$ in $\mZ^d$.
Here the distances for different choices of $a$
 give equivalent topologies,  and we just write $\rho_\infty$ when there's no confusion.
 $\cQ$ is then a Polish space w.r.t. $\rho_\infty$.
 By standard results in topology, the topology induced by  $\rho_\infty$ is equivalent with the product topology on $\cQ$.

We can easily extend the measure $\mu_{\Lambda_{L}, N, \beta}$ to the infinite volume configuration space $\cQ$ by periodic extension, which is still denoted as $\mu_{\Lambda_{L}, N, \beta}$.
Namely, we can construct a random variable with law given by
$\mu_{\Lambda_{L}, N, \beta}$ and extend the random variable periodically,
and the law of the periodic extension gives the desired extension of measure.
 Since $G$ and $\cQ$ are compact,  $\{\mu_{\Lambda_L,N,\beta}\}_{L\ge 1}$ form a tight set.

We will consider the Langevin dynamic on $\cQ$, formally given by
\begin{equ}[e:YM-formal]
	\dif Q = \nabla \mathcal S (Q) \dif t + \sqrt{2}\dif \mathfrak B\;,
\end{equ}
with $\mathfrak B=(\mathfrak B_e)_{e\in E^+}$ being independent Brownian motions on $G$.
This is formal since we will need to ``extend'' $\nabla\cS$ to infinite volume
in a suitable sense. More precisely, the Langevin dynamic we consider is the following SDE system parametrized by $e\in E^+$:
\begin{equs}[eq:YM in]
	\dif Q_e
	&= -\frac12N\beta \sum_{p\in \cP,p\succ e}(Q_p-Q_p^*)Q_e\dif t
	-\frac12(N-1)Q_e\dif t+\sqrt2\dif B_eQ_e ,\qquad \mbox{if} \quad G=SO(N) ,
	\\
	\dif Q_e
	&= -\frac12N\beta \sum_{p\in \cP,p\succ e}\Big( (Q_p-{Q}_p^{*}) - \frac{1}{N}\tr(Q_p-{Q}_p^{*}) I_N\Big)   Q_e\dif t
	\\
	&\qquad\qquad\qquad\qquad\qquad\qquad\quad
	-\frac{N^2-1}NQ_e\dif t+\sqrt2\dif B_eQ_e, \qquad\qquad \mbox{if} \quad G=SU(N).
\end{equs}
 Here $B=(B_e)_{e\in E}$ is a collection of independent Brownian motions on the Lie algebra $\mfg$ of $G$,
and the terms linear in $Q_e$ arise from Casimir elements of the Lie algebras; we will review
these in Section~\ref{sec:Pre}.

\br
We note that the above SDE (in finite volume) was  used earlier in \cite{SSZloop} to derive the
loop equations (i.e. Dyson--Schwinger or Makeenko--Migdal equations) for Wilson loops of the model \eqref{measure}. These loop equations also hold for any infinite volume tight limit of the measures, and in particular for the unique invariant measure for $\beta$ satisfying \eqref{eq:a} as given in Theorem~\ref{th:1.3}.

The study of a quantum field theory of the form \eqref{measure} via a dynamic \eqref{e:YM-formal} 
 is also called stochastic quantization as first proposed by 
\cite{Nelson66,ParisiWu}.
\er


 We will prove that there exists a unique probabilistically strong solutions  to SDE \eqref{eq:YM in} starting from any initial data in $ \cQ$ in  Proposition \ref{lem:4.7}. Hence the solutions form a Markov process in $\cQ$ and the related semigroup is denoted by $(P_t)_{t\geq0}$.

Our first main result is as follows. 

\bt[Uniqueness and ergodicity]
\label{th:1.3}
Under Assumption \ref{ass1}, the following statements hold.

(1) The invariant measure of the Markov semigroup $(P_t)_{t\geq0}$ for the Langevin dynamic 
\eqref{eq:YM in}
 is unique. We denote this invariant measure by $\muYM$.

(2) Furthermore, every tight limit of $\{\mu_{\Lambda_L,N,\beta}\}_L$ is the same,
 and the whole sequence  $\{\mu_{\Lambda_L,N,\beta}\}_L$ converges to $\muYM$ as $L\to \infty$.

(3) Finally, the Markov semigroup $(P_t)_{t\geq0}$ is exponentially ergodic in the following sense: 
 there exists a constant $a>1$ such that
for any $\nu\in \sP(\cQ)$
\begin{equ}[e:WnuP]
	W_2^{\rho_{\infty,a}}(\nu P_t,\muYM)
	\leq C(a) e^{-\widetilde{K}_{\cS} t},\quad t\geq0,
\end{equ}
for some $\widetilde{K}_{\cS}>0$  which only depends on the constant $a$,  $d$, $\beta$ and $G$ (in particular $N$).
\et

Here $W_2^{\rho_{\infty,a}}$ is the Wasserstein distance  w.r.t. $\rho_{\infty,a}$ given for any $\mu, \nu\in \sP(\cQ)$
\begin{align*}
	W_2^{\rho_{\infty,a}}(\mu,\nu)\eqdef \inf_{\pi\in \sC(\mu,\nu)}\pi(\rho_{\infty,a}^2)^{1/2},
\end{align*}
with $\sC(\mu,\nu)$ being the set of couplings between $\mu$ and $\nu$. Remark that $\widetilde{K}_\cS$ can be explicitly given by \eqref{e:tildeK} below and gives a lower bound of spectral gap for $(P_t)_{t\geq0}$ in Wasserstein distance. In Theorem \ref{th:1} we will see that $K_\cS$  gives a lower bound of spectral gap in $L^2(\muYM)$.

\br\label{rem:bc}
The periodic boundary condition in the definition of $\{\mu_{\Lambda_L,N,\beta}\}_L$ is not essential.  By the same argument as in Theorem \ref{th:in1} the tight limit of $\{\mu_{\Lambda_L,N,\beta}\}_L$ when  changing the periodic boundary condition to Dirichlet  or other boundary conditions is also the invariant measure of the SDE \eqref{eq:YM in}, hence, is the same as $\muYM$.
\er

We remark that uniqueness for small $\beta$ could possibly also be proven using the method of Dobrushin, see e.g. \cite{Dobrushin1970}. To this end one would also need to consider the related Wasserstein metric with respect to the  Riemannian distance similarly as we do in this paper. However as we understand such an argument has not been carried out in detail for lattice Yang Mills in the literature. Here  we give a  proof based on a new idea which is a variant of  Kendall--Cranston's coupling used earlier in the stochastic analysis on manifold.

%
%


 The idea for the proof of Theorem \ref{th:1.3} is to use finite dimensional approximation, for which we construct a suitable coupling and find a suitable distance  such that
 the associated Wasserstein distance  between the two finite dimensional approximations starting from different initial distributions decays exponentially fast in time with  {\it uniform} speed.

We define the cylinder functions $ C^\infty_{cyl}(\cQ)$  by
\begin{equ}[e:cylQ]
C^\infty_{cyl}(\cQ)
=\Big\{
F: F=f(Q_{e_1},\dots, Q_{e_n}),n\in \mN, e_i\in E^+, f\in C^\infty(G^n)\Big\}.
\end{equ}
We then obtain the following log-Sobolev inequality for $\muYM$ based on Bakry--\'Emery's criterion.

\bt[Log-Sobolev inequality]
\label{th:1} 
Under Assumption \ref{ass1},
the log-Sobolev inequality holds for the measure $\muYM$ in Theorem \ref{th:1.3}.
Namely, for all cylinder functions $F\in C^\infty_{cyl}(\cQ)$ with $ \muYM(F^2)=1$,
\begin{align}\label{eqL}
	\muYM(F^2\log F^2)\leq \frac2{K_\cS }\sum_{e\in E^+}\muYM(|\nabla_{e} F|^2).
\end{align}
This implies the Poincar\'e inequality, i.e. for all cylinder functions $F\in C^\infty_{cyl}(\cQ)$,
\begin{equ}[e:PoinYM]
	\muYM(F^2)\leq \frac1{K_\cS }\sum_{e\in E^+}\muYM(|\nabla_{e} F|^2)+\muYM(F)^2,
\end{equ}
with $\nabla_{e} $ the gradient w.r.t. the variable $Q_e$.
\et
Theorem \ref{th:1} follows from Theorem \ref{th:1.3} and Corollary \ref{co:lo}, which states the log-Sobolev inequality for every tight limit of $(\mu_{\Lambda_{L},N,\beta})_{L\ge 1}$. The RHS of \eqref{eqL} is the Dirichlet form associated with the Langevin dynamic (see Proposition~\ref{th:D}). Hence, \eqref{eqL} holds for the functions in the domain of Dirichlet form by lower-semicontinuity.

As some simple applications of the  Poincar\'e inequality,
we show certain ``susceptibility'' bounds on the field $Q_e$ and $\tr(Q_p)$, see Corollaries \ref{co:c}
and \ref{co:cc}. These examples demonstrate how to choose suitable functions
in these functional inequalities to yield interesting bounds for the model.

Log-Sobolev and Poincar\'e inequalities in Theorem~\ref{th:1} follow by  checking  the Bakry--\'Emery criteria \cite{MR889476,BakryLedoux}  directly for  the finite dimensional approximation on the product manifolds.
As the Ricci curvatures of the   target manifolds $G$ are given by positive constants,
so are the Ricci curvatures of the configuration space (i.e. the product manifolds).
The Hessian of the Hamiltonian 
could also be bounded by the Ricci curvatures in the strong coupling  regimes.

As a corollary of Theorem~\ref{th:1} we obtain the following large $N$  properties of the Wilson loops.
For the rest of this paper, by a {\it loop} we mean an equivalent class of closed paths (as defined in Section~\ref{sec:YM}), 
where the equivalence relation $\sim$ is given by  cyclic permutations $e_1 e_2 \cdots e_n \sim e_ie_{i+1}\cdots e_ne_1e_2\cdots e_{i-1}$ for any $i\in \{1,...,n\}$,
and it has no two successive edges of the form $e^{-1} e$. We will always assume that a loop is non-empty, i.e. has positive number of edges. 
Given a loop
$\ell = e_1 e_2 \cdots e_n$,
recall that the Wilson loop variable $W_\ell$ is defined as
%
\begin{equ}[e:loop]
W_\ell \eqdef \Tr (Q_{e_1}Q_{e_2}\cdots Q_{e_n})\;.
\end{equ}

\bc[Large $N$ limit of Wilson loops]
\label{co:1} Under Assumption \ref{ass1}, for every Wilson loop \eqref{e:loop},
writing $\var$ and $\E$ for the variance and expectation under the measure $\muYM$  in Theorem \ref{th:1.3},  one has
\begin{equ}[e:varW]
	\var\Big(\frac1NW_\ell\Big)\leq \frac{n(n-3)}{K_\cS N}, \quad G=SO(N);\quad \var\Big(\frac1NW_\ell\Big)\leq \frac{4n(n-3)}{K_\cS N}, \quad G=SU(N) \;.
\end{equ}
In particular, we obtain the convergence
\begin{equ}[e:EW]
\Big|\frac{W_\ell}{N}-\E\frac{W_\ell}{N}\Big|\to0
\qquad
\mbox{as } N\to \infty
\end{equ}
in probability, and  the factorization property of Wilson loops, i.e. for any loops $\ell_1,\dots, \ell_m$
\begin{align*}
	\lim_{N\to\infty}\E\frac{W_{\ell_1}\dots W_{\ell_m}}{N^m}=\lim_{N\to\infty}\prod_{i=1}^m\E\frac{W_{\ell_i}}{N}.
\end{align*}
\ec

Corollary~\ref{co:1} is proven in Section~\ref{sec:3.2}.
Our proof is novel which is 
 based on the Poincar\'e inequality \eqref{e:PoinYM}.
Note that our formulation of the result is different from \cite{Cha} and \cite{Jafar}
in which the factorization property of Wilson loops was obtained by 
 taking a sequence of increasing finite lattices  $\mZ^d = \cup_{N=1}^\infty \Lambda_N$,
considering the correlations of Wilson loops over $\Lambda_N$, 
and taking infinite volume limit simultaneously
as the large $N$ limit when sending $N\to \infty$.
In our approach, we work directly in infinite volume, which seems more natural.
The subtlety here, as mentioned  above  
and also explained in \cite{Cha}, 
is that  the 't Hooft coupling places $N\beta$ instead of $\beta$ in front of the Hamiltonian
so one would require  $N\beta$ to be sufficiently small to obtain the infinite volume limit, which would appear to be problematic
when taking the large $N$ limit afterwards.
However,
thanks to our precise smallness condition on $\beta$ in \eqref{eq:a}, we can take $\beta$ small uniformly in $N$.
This also allows us to
derive bounds on the variances of Wilson loops which are explicit in terms of $N$.  Our proof based on the Poincar\'e inequality which follows from the Bakry--\'Emery condition also appears to be simpler than the arguments in aforementioned previous work.

Furthermore, we obtain the following exponential decay property of the covariance.  Consider $f\in C^\infty_{cyl}(\cQ)$ and we write $\Lambda_f$ for the set of the edges $f$ depends on. Let $|\Lambda_f|$ denote the cardinality of $\Lambda_f$.  We define
\begin{align*}
	\$ f \$_{\infty}\eqdef \sum_{e\in \Lambda_f}\|\nabla_{e}f\|_{L^\infty},
\end{align*}
where $C^\infty_{cyl}(\cQ)$ is introduced in \eqref{e:cylQ} and $\nabla_e$ is introduced in Section \ref{sec:3.1}.
We also write $d(A,B)$ for the distance between $A,B\subset E^+$, which is given by the nearest distance between the vertices in $A$ and $B$.

\bc[Mass gap]
\label{co:zmm} Suppose that Assumption \ref{ass1} holds. Writing $\cov$ for the covariance under the measure $\muYM$ in Theorem \ref{th:1.3}. For $f, g\in C^\infty_{cyl}(\cQ)$, suppose that $\Lambda_f\cap \Lambda_g=\emptyset$. It holds that
\begin{align*}
	\cov(f,g)\leq c_{1} \dg e^{-c_N d(\Lambda_f,\Lambda_g)} \left(\$ f \$_{\infty}\$ g \$_{\infty}+\|f\|_{L^2(\muYM)}\|g\|_{L^2(\muYM)}\right),
\end{align*}
where $c_{1}$ depends on $|\Lambda_f|$, $|\Lambda_g|$ and $c_N$ depends  on $K_\cS$,  $N$ and $d$.
\ec

Note that $f$ and $g$ in the above corollary can be chosen to be Wilson loops, or functions of an arbitrary number of Wilson loops,
which are of particular interest in physics.
We also remark that exponential decay of correlations (together with certain center symmetry conditions)
is also related to Wilson's area law for Wilson loops (see \cite[Theorem~2.4]{MR4278289}).
The proof of Corollary~\ref{co:zmm} is given in Section~\ref{sec:3.2}.

\medskip

We conclude this subsection
by some brief comments on the challenges or subtleties 
in the proofs of the above results.
One of the important ingredients 
in the proofs 
is to estimate  the Hessian or the general second order derivatives for the interaction $\cS$ defined in \eqref{e:defS}.
Note that a term of the form $N\tr(Q_p)$  in the interaction $\cS$ 
would ``appear'' to be of order $N^2$,
which would be too large for us to obtain the desired results.
In our proofs we will properly arrange terms and apply certain properties of the Lie groups and we will show that the relevant second order derivatives are actually at most of order $|\beta|N$. 
See the explanations before Lemma \ref{lem:4.1} and the proof of Theorem \ref{th:1.3} in Section~\ref{sec:uniqueYM} for more details. 
This is one of the crucial reasons which allow us to 
compare the
Ricci curvatures of the Lie groups and Hessians 
to verify the  Bakry--\'Emery condition
by choosing $\beta$ small, and also prove ergodicity 
using a suitable weighted distance.

\subsection{Relevant literature and possible directions}\label{sec:prev}

The study of properties of lattice gauge theories 
recently attracts much interest.
Besides the aforementioned work by \cite{Cha} and \cite{Jafar},
\cite{Chatterjee16}  computed the 
leading terms of free energies, \cite{MR3861073}
provided an elaborate description of loop expectations in the planar setting,
and \cite{ChatterjeeJafar} derived $1/N$ expansions
in the $SO(N)$ case at strong coupling.
Wilson loops (and also Wilson lines when coupled with Higgs) 
for gauge theories with 
finite structure groups were studied in \cite{MR4107931,FLV20,FLV21,Cao20,forsstromAbelian,Adhikari2021};
see also \cite{GS21} for the $U(1)$ case.
Moreover, exponential correlation decay for lattice gauge theories with finite abelian structure groups was obtained by \cite{Forsstrom2021}
using coupling argument, and for  finite non-abelian structure groups
this was proved by \cite{DdhikariCao}.  

Our present article provides a new approach to study these models
via stochastic analysis and dynamical perspective;
see also  \cite{SSZloop} for a new derivation of loop i.e. Makeenko--Migdal  equations for Wilson loops
 by such methods. 

\br
Here by ``stochastic analysis approach'', we do not mean 
the stochastic analysis approach for 2D Yang--Mills in continuum
developed earlier by \cite{GKS89} and \cite{Driver89} (See Def.~3.3 therein) in which parallel translations (which are related with Wilson loops) are formulated as stochastic differential equations.
See \cite{MR3982691} and references therein for more recent literature in this direction.
Our Yang--Mills SDE on the other hand is the stochastic dynamic for 
the connection fields on a lattice with fixed spacing, which is 
along the line of stochastic quantization.
\er

We remark that the choice of constant positive curvature Lie groups $SO(N)$ and $SU(N)$ in this article is a technical simplification for demonstrating our method, 
and it should apply as well  for other compact target spaces with constant or non-constant positive curvatures. 
For instance  it should apply to a lattice $SO(N)$ Yang--Mills model coupled with a 
Higgs field  $\Phi$ which takes values in a sphere in $\R^N$ (i.e. rotator model)
via a gauge-covariant derivative term, whose action takes the form 
$ \Re \sum_{p} \Tr(Q_p) + \sum_e |Q_e \Phi_{v(e)} - \Phi_{u(e)}|^2$.

It would certainly be interesting to show if
log-Sobolev inequalities still hold when the lattice spacing vanishes,
in the situations where the continuum limits of these models are shown or expected to exist. In this direction, on the two dimensional torus, the continuum limit of lattice approximations of the Yang--Mills measures on 1-forms was recently obtained by Chevyrev \cite{Chevyrev19YM},
who  also showed that 
a certain class of Wilson loop observables of this random 1-form 
coincide in law with the corresponding observables under the Yang--Mills measure in the sense of  \cite{Levy03}.
 Note that the Langevin dynamics 
for Yang--Mills models on the two and three dimensional continuous torus were recently constructed in
\cite{CCHS2d,CCHS3d} (see \cite{ChevyrevReview} for a review of these results), and as mentioned in \cite{CCHS2d} it would be interesting to show that the  lattice dynamics of the type \eqref{e:YM-formal} converge
to the processes constructed in the above papers in two and three dimensions.
For some of the recent progress along this direction, 
see the proof of log-Sobolev inequalities for the $\Phi^4_{2,3}$ and sine-Gordon models \cite{RolandPhi4,RolandSG},
and the 1D nonlinear $\sigma$-model 
(see \cite{AnderssonDriver,Hairer16,StringManifold})
for which the log-Sobolev inequality, ergodicity and non-ergodicity (depending on the curvature of the target manifolds) were obtained in \cite{RWZZ17, CWZZ18}.

{\bf Notation}.
Given a Polish space $E$, we write $C([0,T];E)$ for the space of continuous functions from $[0,T]$ to $E$. We use $\sP(E)$ to denote all the probability measures on $E$ with Borel $\sigma$-algebra.

\noindent\textbf{Acknowledgments.}
We would like to thank Scott Smith, Feng-Yu Wang, Xicheng Zhang and Xin Chen
for very helpful discussions.
 H.S. gratefully acknowledges financial support from NSF grants
DMS-1954091 and CAREER DMS-2044415.
 R.Z. gratefully acknowledges financial support from the NSFC (No.
 11922103).  X.Z. is grateful to
the financial supports   by National Key R\&D Program of China (No. 2020YFA0712700) and  the NSFC (No.
12090014, 12288201)  and the support by key Lab of Random Complex Structures and Data Science, Youth Innovation Promotion Association (2020003), Chinese Academy of Science. R.Z. and X.Z. are grateful to the financial support by the DFG through the CRC 1283 ``Taming uncertainty
and profiting from randomness and low regularity in analysis, stochastics and their applications''

\section{Notation and preliminaries}
\label{sec:Pre}

In this section we collect some notation and standard facts about Riemannian geometry, Lie groups and Brownian motions.

\underline{\it Riemannian manifolds.}
 Let $M$ be a Riemannian manifold of dimension $d$.
We denote by $C^\infty(M)$  the space of real-valued smooth functions on $M$.
For $x\in M$ we denote by $T_xM$  the tangent space at $x$ with inner product $\<\cdot,\cdot\>_{T_xM}$.
For $X \in T_xM$, we write $Xf$ or $X(f)$ for the differentiation of $f$ along $X$ at $x$.
For a smooth curve $\gamma:[\alpha,\beta]\to M$ the tangent vector along $\gamma$ is defined by
\begin{align*}
	\dot{\gamma}_tf=\frac{\dif}{\dif t}f(\gamma_t),\quad f\in C^\infty(M).
\end{align*}

Let $\nabla$  be the Levi-Civita connection, which is a bilinear operation associating to vector fields $X$ and $Y$ a vector field $\nabla_{Y}X$. Recall that $(\nabla_{Y}X)(x)$  depends on $Y$ only via $Y(x)$ for $x\in M$ (e.g. \cite[Remark~2.3]{MR1138207}).

For $f\in C^\infty(M)$, we denote by $\nabla f$ the gradient vector field of $f$.  We also write $\Hess(f) $ for the Hessian. 
It can be calculated in the following ways
\begin{equ}[e:hess]
\Hess_f(X,Y)  \eqdef \Hess (f) (X,Y) = \< \nabla_X \nabla f, Y\> = X(Y f ) -  (\nabla_X Y) f\;.
\end{equ}
It is a two-tensor: $\Hess_f(\varphi X,Y)=\Hess_f( X,\varphi Y)
=\varphi  \Hess_f(X,Y)$ for any $\varphi\in C^\infty(M)$ 
so $\Hess_f(X,Y)(x)$ depends only on $X(x)$ and $Y(x)$.
Since Levi-Civita connection is torsion-free, $\Hess (f)$ is symmetric in $X,Y$.

The Riemann curvature tensor $\sR(\cdot,\cdot)$  associated to vector fields $X,Y$ is an operator defined by
\begin{align*}
	\sR(X,Y)Z=\nabla_X(\nabla_YZ)-\nabla_Y(\nabla_XZ)-\nabla_{[X,Y]}Z.
\end{align*}
Let $\{W_i\}_{i=1}^d$ be an orthonormal basis of $T_xM$. The Ricci curvature tensor is defined by
\begin{equ}[e:ricci]
	\Ric(X,Y)=\sum_{i=1}^d\<\sR(X,W_i)W_i,Y\>_{T_xM}
\end{equ}
and is independent of the choice of $\{W_i\}$.
Note that $\Ric(X,Y) (x)$ depends on $X,Y$ only via $X(x),Y(x)$ for $x\in M$.

Let $\gamma$ be a geodesic. A smooth vector field $J$ is called a Jacobi field along $\gamma:[0,t]\to M$ if $\nabla_{\dot\gamma}\nabla_{\dot\gamma}J+\sR(J,\dot\gamma)\dot\gamma=0$. For any $X\in T_{\gamma_0}M$ and $Y\in T_{\gamma_t}M$, there exists a Jacobi field $J$ along $\gamma$ satisfying $J_0=X$ and $J_t=Y$ (c.f. \cite[Section 1.5]{CE75}, \cite[Section 0.4]{Wang} ).




\medskip

\underline{\it Lie groups and algebras.}
For any matrix $M$ we write $M^*$ for the conjugate transpose of $M$.
Let $M_N(\R)$ and $M_N(\C)$ be the space of real and complex $N\times N$ matrices.

For Lie groups $SO(N)$, $SU(N)$,
we write the corresponding Lie algebras
as $\so(N)$, $\su(N)$ respectively.
Every matrix $Q$ in one of these Lie groups satisfies $QQ^* = I_N$, and every matrix $X$ in one of these Lie algebras
satisfies $X + X^* = 0$.
Here $I_N$ denotes the identity matrix.

We endow $M_N(\C)$ with the Hilbert-Schmidt inner product
\begin{equ}[e:HS]
	\< X,Y\> = \Re \Tr (X Y^*)
	\qquad  \forall X,Y\in M_N(\C).
\end{equ}
We restrict this inner product to our
Lie algebra $\mfg$, which is then invariant under the adjoint action.
In particular for $X,Y\in \so(N)$ or $\su(N)$ we have
$\< X,Y\> =- \Tr (XY)$. 
Note that $\tr(XY) \in \R$ since 
we have $\tr ((XY)^*) = \tr(Y^* X^*) = \tr(XY)$,
and $\tr(A^*)= \overline{\tr(A)}$  for any $A\in M_N(\C)$.

Below $G$ is always understood as $SO(N)$ or $SU(N)$.
Every $X\in \mfg$ induces a right-invariant vector field  $\widetilde X$
on $G$, and for each $Q\in G$, $\widetilde X(Q)$  is just given by $XQ$ since $G$ is a matrix Lie group. Indeed, given any $X \in \mfg$, the curve $t \mapsto e^{tX} Q$ is well approximated near $t=0$ by $Q+tXQ$ up to an error of order $t^{2}$.

The inner product on $\mfg$ induces an inner product
on the tangent space at every $Q\in G$
via the right multiplication on $G$.
Hence, for $X,Y\in \mfg$, we have $XQ,YQ \in T_{Q}G$, and their inner product is given by $\Tr((XQ)(YQ)^*) = \Tr(XY^*)$.
This yields a bi-invariant Riemannian metric on $G$.

For any function $f\in C^\infty (G)$ and $X \in \mfg$,
the right-invariant vector field $\widetilde X$ induced by $X$ acts on $f$ at $Q\in G$ by the right-invariant
derivative
\begin{equ}[e:Xf]
\widetilde X f (Q)=  \frac{\dif}{\dif t}|_{t=0} f(e^{tX} Q).
\end{equ}
We have
\begin{equ}
\widetilde{[X,Y]} = [\widetilde X, \widetilde Y] ,
\qquad \mbox{namely,} \qquad
([X,Y] Q ) f (Q) = [XQ, YQ] f(Q),
\end{equ}
where the $[\cdot,\cdot]$   is the Lie bracket on $\mfg$ on the LHS and  the vector fields commutator  on the RHS.
Also, for the Levi-Civita connection $\nabla$ we have
\begin{equ}[e:F27]
\nabla_{\widetilde X} ( \widetilde Y) =\frac12 \widetilde{ [X,Y]  }.
\end{equ}

We refer the above facts to \cite[Appendix~F]{AGZ}, e.g. Lemma F.27 therein.

%

\medskip

\underline{\it Brownian motions.}
Denote by $\mathfrak{B}$ and $B$ the Brownian motions on a Lie group $G$ and its Lie algebra $\mfg$ respectively.
The Brownian motion $B$ is characterized by
\begin{equation}\label{eq:DefBM}
\E \Big[  \<B(s),X \> \<B(t),Y \>  \Big] = \min(s,t) \< X,Y \>
\qquad
\forall
X,Y \in \mfg.
\end{equation}
By  \cite[Sec.~1.4]{Levy11}, the Brownian motions $\mathfrak{B}$ and $B$ are related through the following SDE:
\begin{equ}[e:dB]
	\dif \mathfrak B = \dif B \circ \mathfrak B = \dif B\, \mathfrak B
	+ \frac{c_\mfg}{2} \mathfrak B \dif t,
\end{equ}
where $\circ$ is the Stratonovich product, and $\dif B\, \mathfrak B$ is in the It\^o sense.
Here the constant $c_\mfg$ is determined by
$ \sum_{\alpha} v_\alpha^2  =c_\mfg I_N$ where
$(v_\alpha)_{\alpha=1}^{\dg}$ is an orthonormal basis of $\mfg$.
Moreover, by \cite[Lem.~1.2]{Levy11},
\footnote{Note that in \cite[Lem.~1.2]{Levy11}, the scalar product differs from \eqref{e:HS} by a scalar multiple depending on $N$ and $\mfg$, so we accounted for this in the expression for $c_\mfg$ above.}
\begin{equ}[e:c_g]
	c_{\so(N)} =  -\frac12(N-1),
	\quad
	c_{\su(N)} =  -\frac{N^2-1}{N} .
\end{equ}

Denote by $\delta$  the Kronecker function, i.e. $\delta_{ij}=1$ if $i=j$ and $0$ otherwise. For any matrix $M$, we write $M^{ij}$ for its $(i,j)$th entry.
The following holds by straightforward calculations (see e.g. \cite[(2.5)]{SSZloop}):
\minilab{e:BB}
\begin{equs}[2]
	\dif\< B^{ij}, B^{k\ell}\>
	&=\frac12(\delta_{ik}\delta_{j\ell}-\delta_{i\ell}\delta_{jk})\dif t,
	&\qquad \mfg=\so(N);		\label{e:BB1}
	\\
	\dif\< B^{ij}, B^{k\ell}\>
	&= \big( -\delta_{i\ell} \delta_{jk} + \frac{1}{N} \delta_{ij}\delta_{k\ell} \big)\, \dif t\;,
	&\qquad\mfg=\su(N).				\label{e:BB3}
\end{equs}

\subsection{Product manifolds and Lie groups}
\label{sec:product}

For Riemannian manifolds $M_1, M_2$,
the tangent space of the product manifolds
$T_{(x_1,x_2)} (M_1 \times M_2)$ is isomorphic with $T_{x_1} M_1 \oplus T_{x_2} M_2$ which is endowed with
the inner product
$$\< u_1+u_2,v_1+v_2\>_{T_{(x_1,x_2)} (M_1 \times M_2)} = \< u_1,v_1\>_{T_{x_1} M_1}+\< u_2,v_2\>_{T_{x_2} M_2}.$$
For a finite collection of Riemannian manifolds $(M_e)_{e\in A}$ where $A$ is some finite set, 
the product is defined analogously.

If all $M_e$ are the same manifold $M$, the product is written as $M^{A}$.
In this case, given a point $x= (x_e)_{e\in A} \in M^A$,
 if $u_e \in T_{x_e} M_e$ for some $x_e \in M_e$,
we will sometimes view $u_e$ as a tangent vector in $T_x M^{A}$
which has zero components for all $\bar e \neq e$.
Continuing with this notation,
if $(v_e^i)_{i=1}^d$ is a basis (resp. orthonormal basis) of $T_{x_e} M_e$, then
 $(v_e^i)_{e\in A}^{i=1,...,d}$ is a basis (resp. orthonormal basis) of $T_{x} M^A$.

For Lie groups $G_1,G_2$, the group multiplication is defined on $G_1\times G_2$ componentwise.
The Lie algebra $\mfg$ of $G_1\times G_2$ is isomorphic to $\mfg_1 \oplus \mfg_2$ where $\mfg_i$ is the Lie algebra of $G_i$.
The Lie bracket on  $\mfg_1 \oplus \mfg_2$  is defined componentwise.
 If $X=(X_1,X_2) \in \mfg$,  then induced the right-invariant vector field
$\widetilde{X}(x)$ for every $x\in G_1\times G_2$
is equal to $(\widetilde X_1(x),\widetilde X_2(x))$. In particular,
\eqref{e:F27} still holds for any two right-invariant vector fields  on the  Lie group product.

With similar notation as above we can define product $G^A$ and its Lie algebra $\mfg^A$ for a finite set $A$.
 Given $X \in \mfg^A$, the exponential map $t \mapsto \text{exp}(tX)  $ is also defined pointwise as
 $\text{exp}(tX)_{e}\eqdef e^{tX_{e}}$ for each $e\in A$.

In the following we choose $G$ to be one of the matrix Lie groups as before.
Define the configuration space as the Lie group product  $\cQ_L=G^{E_{\Lambda_{L}}^+}$,
consisting of all maps $Q:e \in E_{\Lambda_{L}}^{+} \mapsto Q_{e} \in G$.
Let $\q_L=\mfg^{E_{\Lambda_{L}}^+}$ be the corresponding direct sum of $\mfg$.
Note that $\q_L$ is the Lie algebra of the Lie group $\cQ_L$.
For any matrix-valued functions $A,B$ on $E_{\Lambda_L}^+$, we denote by $AB $
the pointwise product $(A_e B_e)_{e\in E_{\Lambda_{L}}^+}$.

As above, the tangent space at $Q \in \mathcal{Q}_L$ consists of the products $XQ= (X_e Q_e)_{e\in E_{\Lambda_{L}}^+}$ with $X \in \q_L$,  and given two such elements $XQ$ and $YQ$, their inner product is defined by
\begin{equation}
	\langle XQ,YQ \rangle_{T_Q\cQ_L} = \sum_{e \in E_{\Lambda_{L}}^{+} } \text{Tr}(X_{e}Y_{e}^{*}) \nonumber.
\end{equation}
The basis of the tangent space $T_Q\cQ_L$ is given by $\{X_e^iQ : e\in E^+_{\Lambda_L} , 1 \le i \le \dg \}$
where for each $e$, $\{X_e^i\}_i$ is a basis for $\mfg$.

Given any function $f \in C^{\infty}(\mathcal{Q}_L)$, the right-invariant derivative is given by $\frac{\dif}{\dif t}|_{t=0} f(\text{exp}(tX)  Q)$.  For each $Q \in \mathcal{Q}_L$, the gradient $\nabla f(Q)$ is an element of the tangent space at $Q$ which satisfies for each $X \in \q_L$
\begin{equation}
	\langle \nabla f(Q),  XQ \rangle_{T_Q\cQ_L} = \frac{\dif}{\dif t}\Big|_{t=0} f(\text{exp}(tX)  Q)=(XQ)f.
\end{equation}
We can write $\nabla f=\sum_{i=1}^{\dg}\sum_{e\in E_{\Lambda_L}^+}(v^i_ef)v^i_e$
with $\{v^i_e : e\in E_{\Lambda_L}^+, i=1,\cdots, \dg\}$ being an orthonormal basis of $T_Q\cQ_L$.
We then define
\begin{align*}
	\nabla_{e}f\eqdef \sum_{i=1}^{\dg}(v_e^if)v_e^i,
	\qquad
	\Delta_ef\eqdef  \div\nabla_ef=\sum_{i=1}^{\dg}\<\nabla_{v_e^i}\nabla_{e}f,v_e^i\>.
\end{align*}
Here $\nabla_{e}$ and $\Delta_{e}$ can be viewed as  the gradient and the  Laplace--Beltrami operator (w.r.t. the variable $Q_e$) on
$G$ endowed with the metric given above.  


\section{Yang--Mills SDE}\label{sec:3.1}
In this section we first recall the Langevin dynamics \eqref{eq:YM in} associated to the lattice Yang--Mills model in finite volume from \cite{SSZloop}. 
We then extend the dynamics from finite volume to the whole $\mZ^d$  and prove the global well-posdness of the SDE \eqref{eq:YM in}. Furthermore, we prove that every tight limit of $(\mu_{\Lambda_L,N,\beta})_L$ is an invariant measure for the SDE \eqref{eq:YM in}.

We consider the Langevin dynamic for the measure \eqref{measure}, which is
the following SDE on $\cQ_L$
\begin{equ}[eq:YM1]
	\dif Q = \nabla \mathcal S (Q) \dif t + \sqrt{2}\dif \mathfrak B\;,
\end{equ}
with $\mathfrak B=(\mathfrak B_e)$ being independent Brownian motions on $G$.
Here $\dif \mathfrak B$ can be viewed as the white noise w.r.t. the inner product on $T_Q\cQ_L$.

We now recall the explicit expression for $ \nabla \mathcal S$. To this end, we introduce the following notation. For a plaquette $p=e_1e_2e_3e_4 \in \CP$, we write $p\succ e_1$ to indicate that $p$ is a plaquette that starts from edge $e_{1}$.
Note that for each edge $e$, there are $2(d-1)$ plaquettes in $\CP$
such that $p\succ e$.
For {\it any} Lie algebra $\mfg$ embedded into $M_N(\C)$ (denoted as $M_N$ below for short), it forms a closed subspace of $M_N$,
 and therefore $M_N$ has an orthogonal decomposition $M_N=\mfg \oplus \mfg^\perp$.
Given $M\in M_N$, we denote by $\proj M \in \mfg$ the orthogonal projection onto $\mfg$.
\begin{lemma}\label{lem:gradS}
Writing $\cdot$ for matrix multiplication, for each $e \in E_{\Lambda_{L}}^{+}$ we have
\begin{equ}[e:grad]
\nabla\cS(Q)_e=  N\beta \sum_{p\in \cP_{\Lambda_{L}}, p\succ e} \proj Q_{p}^{*} \cdot (Q_{e}^{*})^{-1}\;.
\end{equ}
\end{lemma}

\begin{proof}
See \cite[Lemma 3.1]{SSZloop}. We remark that in this calculation of the gradient of $\cS$, for each fixed $e$,
we replace the $Q_p$ in \eqref{e:defS} where $p$ contains $e$ or $e^{-1}$
by a product of the form $Q_e Q_\cdot Q_\cdot Q_\cdot$, which does not change the trace.
This motivates our introduction of the notation $p\succ e$.
\end{proof}

The above result holds for general matrix Lie groups, and
for our specific choices of Lie groups, we have the  SDE system \eqref{eq:YM1} on the finite lattice $\Lambda_L$
more explicitly as
\begin{equ}[eq:YM]
	\dif Q_e = \nabla \mathcal S (Q)_e \dif t
	+ c_{\mfg} Q_e \dif t + \sqrt{2} \dif B_e Q_e \;,
	\qquad  (e\in E_\Lambda^+)
\end{equ}
\begin{equ}[e:DS]
 \nabla\cS(Q)_e=
\begin{cases}
 \displaystyle
 -\frac12N\beta \sum_{p\in \cP_{\Lambda_L},p\succ e}(Q_p-Q_p^*)Q_e\;,
&\qquad
G=SO(N)\;,
\\
 \displaystyle - \frac12 N\beta \sum_{p\in \cP_{\Lambda_L} ,p\succ e}
	\Big( (Q_p-{Q}_p^{*}) - \frac{1}{N}\tr(Q_p-{Q}_p^{*}) I_N\Big)   Q_e\;,
&\qquad
G = SU(N)\;.
\end{cases}
\end{equ}

We recall the following two results from \cite[Lemmas 3.2-3.3]{SSZloop}.

\bl\label{lem:exist} For fixed $N\in\mN$ and any initial data $Q(0)=(Q_e(0))_{e\in E_{\Lambda_L}^+}\in \cQ_L$, there exists a  unique solution $Q=(Q_e)_{e\in E_{\Lambda_L}^+}\in C([0,\infty);\cQ_L)$ to \eqref{eq:YM}.
\el

\begin{lemma}\label{lem:inv}
	\eqref{measure} is invariant under the SDE system \eqref{eq:YM}.
\end{lemma}

By global well-posedness of the SDE \eqref{eq:YM}, the solutions form a Markov process in $\cQ_L$. We use $(P_t^L)_{t\geq0}$ to denote the associated semigroup, i.e. for $f\in C^\infty(\cQ_L)$, $(P_t^Lf)(x)=\E f(Q(t,x))$ for $x\in \cQ_L$, where $Q(t,x)$ denotes the solution at time $t$ to \eqref{eq:YM} starting from $x\in \cQ_L$. 
We can also write down the Dirichlet form associated with $(P_t^L)_{t\geq0}$. More precisely, for $F\in C^\infty(\cQ_L)$ we consider the following  symmetric quadratic form
\begin{align*}
	\cE^L(F,F)&\eqdef\int \<\nabla F,\nabla F\>_{T_Q\cQ_L}\dif \mu_{\Lambda_{L},N,\beta}
	\\&=\sum_{e\in E_{\Lambda_{L}}^+}\int\<\nabla_{e}F,\nabla_{e} F\>\dif \mu_{\Lambda_{L},N,\beta}
	\\&=\sum_{e\in E_{\Lambda_{L}}^+}\int\tr(\nabla_{e}F(\nabla_{e} F)^*)\dif \mu_{\Lambda_{L},N,\beta}.
\end{align*}
Using integration by parts formula for the Haar measure,
we have that $(\cE^L, C^\infty(\cQ_L))$ is closable,  and its closure $(\cE^L,D(\cE^L))$ is a regular Dirichlet form on $L^2(\cQ_L,\mu_{\Lambda_L,N,\beta})$. (c.f. \cite{Fukushima}.)

Recall $\cS $ in \eqref{e:defS}.
We  write the generator associated to the above Dirichlet form for $F\in C^\infty(\cQ_L)$ as
\begin{align}\label{eq:L}
	\cL_L F=	\sum_{e\in E_{\Lambda_{L}}^+}\Delta_{e}F+\sum_{e\in E_{\Lambda_{L}}^+}\<\nabla \cS(Q)_e,\nabla_{e} F\>.
\end{align}
We  use $D(\cL_L)$ to denote the domain of the generator. Moreover, $\cE^L(F,G)=-\int \cL_L FG\dif \mu_{\Lambda_{L},N,\beta}$ for $F,G\in C^\infty(\cQ_L)$. 
It is easy to see that $(P_t^L)_{t\geq0}$ is a $\mu_{\Lambda_{L},N,\beta}$-version of the $L^2(\mu_{\Lambda_L,N,\beta})$-semigroup associated with the Dirichlet form $(\cE^L,D(\cE^L))$.  (c.f. \cite{MaRockner} or \cite{Fukushima}.)

Recall that $\cQ=G^{E^+}$. Now we extend $\Lambda_L$ to $\mZ^d$ and consider the  SDE \eqref{eq:YM in} on the entire space. To this end, we write $M_N^{E^+} = \prod_{e\in E^+} M_N$ for the direct product of infinitely many vector spaces $M_N$
(i.e. an element of $M_N^{E^+}$ is allowed to have infinitely many non-zero components).
We define  a norm  on $M_N^{E^+}$  by
\begin{align}\label{eq:norm}
	\|Q\|^2\eqdef \sum_{e\in E^+}\frac1{2^{|e|}}|Q_e|^2,
\end{align}
with $|Q_e|^2=\<Q_e,Q_e\>$ for $\<\cdot,\cdot\>$ as in \eqref{e:HS}  and $|e|$ given by the distance from $0$ to $e$ in $\mZ^d$.
(More precisely, $|e|$ is the minimum of the distances from the two vertices of $e$ to $0$.) Now we give existence and uniqueness of solutions to \eqref{eq:YM in}. 

\bp\label{lem:4.7}
Fix $N\in\mN, \beta\in\mR$. For any $Q^0\in \cQ$, there exists a unique probabilistically strong solution $Q$ to \eqref{eq:YM in} in $C([0,\infty);\cQ)$ starting from any $Q^0$. Namely, for a given probability space  $(\Omega,\cF, \mP)$ and  Brownian motion $(B_e)_{e\in E^+}$ on it, there exists an $(\cF_t)_{t\geq0}$-adapted process $Q\in C([0,\infty);\cQ)$  and $Q$ satisfy \eqref{eq:YM in} $\mP$-a.s. with $(\cF_t)_{t\geq0}$ given by normal filtration generated by the Brownian motion $(B_e)_{e\in E^+}$.
\ep

\begin{proof}
For every initial data $Q^0\in \cQ$,
 we can easily find $Q^L(0)\in \cQ_L$ with the periodic extension still denoted by $Q^L(0)$
  such that $\|Q^L(0)-Q^0\|\to0$ as $L\to\infty$.
Indeed, we can set $Q^L(0)$ as follows (the specification on $E^+_{\Lambda_L}\backslash E^+_{\Lambda_{L-1}}$ is just to ensure periodic boundary condition):
\begin{equ}
	Q_e^L(0)=
	\begin{cases}
		\displaystyle
		Q_e(0)&\qquad
		e\in E^+_{\Lambda_{L-1}}
		\\
		\displaystyle I_N
		&\qquad
		e\in E^+_{\Lambda_L}\backslash E^+_{\Lambda_{L-1}}\;.
	\end{cases}
\end{equ}
By Lemma \ref{lem:exist} we obtain a unique solution $Q^L\in C([0,\infty);\cQ_L)$ to \eqref{eq:YM} from $Q^L(0)$. We could also extend $Q^L$ to $\cQ$ by periodic extension.
Since $\cQ$ is compact,
 the marginal laws of $\{Q^L\}$ at each $t\ge 0$ form a tight set in $\cQ$.

 Furthermore,  using the SDE \eqref{eq:YM}, for $t\geq s\geq0$
\begin{align*}
	Q_e^L(t)-Q_e^L(s)=\int_s^t\nabla \cS(Q^L)_e\dif r+\int_s^tc_{\mfg}Q_e^L\dif r+\sqrt2\int_s^t\dif B_eQ_e^L.
\end{align*}
By It\^o's formula and the fact that
$Q_e^L\in G$ which is {\it compact}, we have the following bound for $p\geq 1$ and $0\leq s,t\leq T$
\begin{align*}
	\E|Q_e^L(t)-Q_e^L(s)|^{2p}\leq C_{N,\beta,p,T}(|t-s|^{2p}+|t-s|^p),
\end{align*}
where $C_{N,\beta,p,T}$ is a positive constant and may change from line to line.
Since the above constant is independent of $e$, we have
\begin{align*}
	\E\|Q^L(t)-Q^L(s)\|^{2p}\leq C_{N,\beta,p,T}(|t-s|^{2p}+|t-s|^p).
\end{align*}
Hence, by Kolmogorov criterion we have for $\alpha<1/2$
\begin{align*}
	\sup_L\E\Big(\sup_{s\neq t\in [0,T]}\frac{\|Q^L(t)-Q^L(s)\|}{|t-s|^\alpha}\Big)<\infty.
\end{align*}
Hence, the laws  of $\{Q^L\}$, which are denoted by $\{\mP^L\}$ form a tight set in $C([0,\infty);\cQ)$ equipped with the distance
\begin{align*}
	\widetilde\rho(Q,Q')\eqdef \sum_{n=0}^\infty 2^{-n}\Big(1\wedge \sup_{t\in [n,n+1]}\|Q(t)-Q'(t)\|\Big), \quad Q,Q'\in C([0,\infty);\cQ).
\end{align*}
We write  $\mP^Q$ for one tight limit. For simplicity we still write  $\{\mP^L\}$ for the converging subsequence.
Since $(C([0,\infty);\cQ), 	\widetilde\rho)$ is a Polish space, existence follows from the usual Skorohod Theorem and
taking limit on the both sides of the equation.
More precisely, there exists a stochastic basis $(\tilde \Omega,\tilde\sF,\tilde \mP)$ and $C([0,\infty);\cQ)$-valued random variables $\{\tilde Q^L\}_L, \tilde Q$ on it such that $\tilde Q^L \stackrel{d}{=} \mP^L$, 
$\tilde Q \stackrel{d}{=} \mP^Q$ and $\tilde Q^L\to \tilde Q$ in $C([0,\infty);\cQ)$ $\tilde \mP$-a.s., $L\to\infty$. 
 As a result, for every  $F \in C^\infty_{cyl}(\cQ)$, which can be viewed as function on $\cQ_L$ for $L$ large enough, we know that
\begin{align*}
	F(\tilde Q^L(t))-F(Q^L(0))-\int_0^t\cL_L F(\tilde Q^L(s))\dif s
\end{align*}
is a $\tilde\mP$-martingale, where $\cL_L$ is as in \eqref{eq:L}.
Letting $L\to \infty$
\begin{align*}
	F(\tilde Q(t))-F(Q^0)
	-\sum_{e\in E^+}\int_0^t \Big(\Delta_{e}F(\tilde Q(s))+\<Z_e,\nabla_{e}F\>(\tilde Q(s))\Big)\dif s
\end{align*}
is a $\tilde\mP$-martingale with 
\begin{equ}\label{def:Ze}
	Z_{e}(Q)\eqdef
	\begin{cases}
		\displaystyle
		-\frac12N\beta \sum_{p\in \cP,p\succ e}(Q_p-Q_p^*)Q_e\;,
		&\qquad
		G=SO(N)\;,
		\\
		\displaystyle - \frac12 N\beta \sum_{p\in \cP ,p\succ e}
		\Big( (Q_p-{Q}_p^{*}) - \frac{1}{N}\tr(Q_p-{Q}_p^{*}) I_N\Big)   Q_e\;,
		&\qquad
		G=SU(N)\;.
	\end{cases}
\end{equ}
We then obtain a martingale solution to \eqref{eq:YM in}.
By martingale representation theorem we could construct a stochastic basis and on it Brownian motions $(\overline B_e)_{e\in E^+}$ and ${\overline Q}\in C([0,\infty),\cQ)$ with law given by $\mP^Q$ such that $\overline Q$ and $(\overline B_e)_{e\in E^+}$ satisfy SDE  \eqref{eq:YM in}, which gives the existence of probabilistically weak solutions to SDE  \eqref{eq:YM in}.

Now we prove pathwise uniqueness:
Consider two solutions  $Q, Q'\in C([0,T];\cQ)$ starting from the same initial data $Q(0)\in \cQ$ and we apply It\^o's formula to calculate $\dif \|Q-Q'\|^2$. Since $Q_e, Q_e'\in G$ for every $e\in E^+$,
by the Burkholder--Davis--Gundy inequality  and \eqref{e:BB} for the stochastic integral we obtain
\begin{align*}
\E & \sup_{t\in [0,T]}|Q_e-Q'_e|^2
\\
&\leq C_{N,\beta,T}\int_0^T\E|Q_e-Q'_e|^2\dif s+C_{N,\beta,T}\sum_{p\in \cP,p\succ e}\sum_{\bar e\in p}\int_0^T\E|Q_e-Q'_e||Q_{\bar e}-Q'_{\bar e}|\dif s,
\end{align*}
where $C_{N,\beta,T}$ may change from line to line.
We then use
\begin{align*}
	2|Q_{\bar e}-Q'_{\bar e}||Q_e-Q'_e|\leq |Q_{\bar e}-Q'_{\bar e}|^2+|Q_e-Q'_e|^2,
\end{align*}
to obtain
\begin{align*}
	\frac1{2^{|e|}}\E\sup_{t\in [0,T]}|Q_e-Q'_e|^2
	&\leq C_{N,\beta,T} \frac1{2^{|e|}}\int_0^T\E|Q_e-Q'_e|^2\dif s
	\\
	&+C_{N,\beta,T}\sum_{p\in \cP, p\succ e}\sum_{e\neq \bar e\in p}\frac1{2^{|\bar e|}}\int_0^T\E|Q_{\bar e}-Q'_{\bar e}|^2\dif s.
\end{align*}
Summing over $e$ we get
	\begin{align*}
		\E\sup_{t\in [0,T]}\|Q-Q'\|^2\leq C_{N,\beta,T} \int_0^T\E\|Q-Q'\|^2\dif s.
	\end{align*}
Hence, pathwise uniqueness follows by Gronwall's lemma.
By Yamada--Watanabe Theorem \cite{Kurtz},
weak existence and pathwise uniqueness gives us  existence and uniqueness of strong solution.
In particular one can consider other boundary conditions for finite $L$ and the infinite volume limit solution is the same which is the unique solution to SDE \eqref{eq:YM in}.
\end{proof}

By Proposition \ref{lem:4.7}, the solutions to \eqref{eq:YM in} form a Markov process in $\cQ$.
We denote by $(P_t)_{t\geq0}$ the associated semigroup.
 As we are in the compact setting, it is easy to obtain the tightness of the field  $(\mu_{\Lambda_{L},N,\beta})_L$ in $\cQ$ as $L\to \infty$. Since by Lemma \ref{lem:inv} $\mu_{\Lambda_{L},N,\beta}$ is an invariant measure for \eqref{eq:YM}, we then obtain the following result.

\bt\label{th:in1} Every tight limit $\mu_{N,\beta}$ of $\{\mu_{\Lambda_{L},N,\beta}\}$ is an invariant measure for \eqref{eq:YM in}.
\et

\begin{proof}
Suppose that a subsequence -- still denoted by $\mu_{\Lambda_{L},N,\beta}$ for simplicity -- converges to $\mu_{N,\beta}$ weakly in $\cQ$.
	We start from the unique solutions $Q_L$ to equation \eqref{eq:YM} with initial distribution $\mu_{\Lambda_{L},N,\beta}$ and the unique solutions $Q$ to \eqref{eq:YM in}  with initial distribution $\mu_{N,\beta}$.
	By exactly the same arguments as in the proof of Proposition \ref{lem:4.7}, we know that the laws of $\{Q_L\}$ are also tight in $C([0,T];\cQ)$ and every tight limit satisfies equations \eqref{eq:YM in}  with initial distribution $\mu_{N,\beta}$.
	By  uniqueness of solution to \eqref{eq:YM in}  from Proposition \ref{lem:4.7},
	we have that the whole sequence of the laws of $\{Q_L\}$ converge to the law of $Q$ in $C([0,T];\cQ)$. Since by Lemma \ref{lem:inv} $\mu_{\Lambda_{L},N,\beta}$ is an invariant measure for \eqref{eq:YM}, the result follows.
\end{proof}

We can also write the Dirichlet form and generator associated with \eqref{eq:YM in}. 
Recall $ C^\infty_{cyl}(\cQ)$ defined in \eqref{e:cylQ}.
For every tight limit $\mu_{N,\beta}$ and $F\in C^\infty_{cyl}(\cQ)$ we define
the following  symmetric quadratic form
\begin{equs}[e:EFF]
	\cE^{\mu_{N,\beta}}(F,F)&\eqdef\sum_{e\in E^+}\int\<\nabla_{e}F,\nabla_{e} F\>\dif \mu_{N,\beta}
	\\
	&=\sum_{e\in E^+}\int\tr(\nabla_{e}F(\nabla_{e} F)^*)\dif \mu_{N,\beta}.
\end{equs}
By \eqref{eq:L} and letting $L\to\infty$ we have that for $F, G\in C^\infty_{cyl}(\cQ)$
\begin{align}\label{eq:E}
	\cE^{\mu_{N,\beta}}(F,G)=-\int \cL F\, G\, \dif \mu_{N,\beta},
\end{align}
with
\begin{align*}
	\cL F\eqdef\sum_{e\in E^+}\Delta_{e}F+\sum_{e\in E^+}\<Z_e,\nabla_{e} F\>,
\end{align*}
for $Z_e$ in \eqref{def:Ze}.

\bp\label{th:D}
$(\cE^{\mu_{N,\beta}},C^\infty_{cyl}(\cQ))$ is closable and its closure $(\cE^{\mu_{N,\beta}},D(\cE^{\mu_{N,\beta}}))$ is a regular Dirichlet form on $L^2(\cQ,\mu_{N,\beta})$.
\ep

\begin{proof}
It is sufficient to prove the closability of $(\cE^{\mu_{N,\beta}},C^\infty_{cyl}(\cQ))$.
Let $\{F_n\}_{n\in \mN}\subset C^\infty_{cyl}(\cQ)$ be such that
\begin{align*}
\lim_{n\to\infty}\mu_{N,\beta}(F^2_n)=0,
\qquad \lim_{n,m\to\infty}\cE^{\mu_{N,\beta}}(F_n-F_m,F_n-F_m)=0.
\end{align*}
Using \eqref{eq:E}, for $G\in C^\infty_{cyl}(\cQ)$ we have
\begin{align*}
	\cE^{\mu_{N,\beta}}(G,F_n)=-\int \cL G \, F_n \, \dif \mu_{N,\beta}\to0.
\end{align*}
Hence, the result follows from \cite[Chapter I. Lemma 3.4]{MaRockner}.
\end{proof}

\section{Log-Sobolev and Poincar\'e inequalities and applications}\label{sec:3.2}

In this section we prove log-Sobolev inequality under the usual Bakry--\'Emery condition (see \eqref{c1} below). As applications we obtain large $N$ limit, factorization property of rescaled Wilson loops and the exponential decay of a large class of observables. 

\subsection{Log-Sobolev and Poincar\'e inequalities}

In this section we first prove Log-Sobolev and Poincar\'e inequalities for the probability measure $\mu_{\Lambda_L,N,\beta}$ on the finite dimensional compact manifold $\cQ_L$. We then let $L\to\infty$ to derive the log-Sobolev inequality for every tight limit of $\mu_{\Lambda_L,N,\beta}$. As simple application we give a proof of correlation bounds  (or susceptibility bounds) for the field $Q$, and
that for the ``microscopic Wilson loops'' $\tr(Q_p)$ for plaquettes $p$. 

Below to verify  the Bakry--\'Emery's condition we  need to calculate
$\Hess_{\cS}(v,v)(Q)$ and $\Ric(v,v)(Q)$ for  $v\in T_Q\cQ_L$ and $Q\in \cQ_L$, 
and we recall \eqref{e:hess}\eqref{e:ricci} for their definitions.
Following the convention in Section~\ref{sec:product},
we write
\begin{equ}[e:v]
v=(v_e)_{e\in E^+_\Lambda}=\sum_{e\in E_\Lambda^+}X_eQ_e
\end{equ}
with $X_e \in \q_L$ being zero for all components except for the component $e$. We also write $|v|^2=\<v,v\>_{T_Q\cQ_L}$.

We first compute $\Hess_{\cS}(v,v)$ for $v\in T_Q \cQ_L$.
Note that as a ``naive'' guess,
 $\cS$ defined in \eqref{e:defS} would appear to 
be of order $N^2$, since the trace of the orthogonal or unitary matrix $Q_p$ would be generally bounded by $N$ 
and there is another factor $N$ outside the summation.
If the Hessian of $\cS$ was indeed of order $N^2$, or $N^p$ for any $p>1$,
then  in 
Assumption \ref{ass1} we would never be able to 
fix $\beta$ small uniformly in $N$ and ensure that $K_\cS$ is strictly positive when $N$ gets large. Fortunately in the next lemma by properly arranging terms and using H\"older inequalities we prove that the Hessian is actually at most of order $N$. 

\bl\label{lem:4.1} 
For $v=XQ\in T_Q\cQ_L$ we have
\begin{align}\label{eq:hess}
	|\Hess_{\cS}(v,v)| \leq 8(d-1)N|\beta||v|^2.
\end{align}
\el

\begin{proof}
 In the proof we omit the subscript $L$ for simplicity. We view $v = XQ\in T_Q\cQ_L$ as a right-invariant vector field on $\cQ_L$ generated
by $X \in \q_L$. By \eqref{e:F27} (and the discussion in Sec.~\ref{sec:product}) we have $\nabla_v v =0$. We apply the second identity in \eqref{e:hess} and using \eqref{e:v} we have
\begin{equ}[e:XQXQ]
\Hess_{\cS}(v,v) = v(v(\cS)) = \sum_{e,\bar e\in E_\Lambda^+ } ( X_{\bar e}Q_{\bar e})(X_eQ_e)\cS.
\end{equ}
Recall $\cS$ from \eqref{e:defS}, which is a sum
over  plaquettes $p \in \CP^+_{\Lambda_{L}}$, and is also
a linear (affine) function in each variable $Q_e$.
Since $p\in \cP^+$, we can write $Q_p = Q_{e_1}Q_{e_2}Q_{e_3}^*Q_{e_4}^*$ where $e_1,e_2,e_3,e_4\in E^+$.
Then it is easy to calculate the derivatives using the definition \eqref{e:Xf}, for instance
\begin{equ}[e:X3Q]
( X_{e_3} Q_{e_3} ) Q_p
= \frac{\dif}{\dif t}\Big|_{t=0} Q_{e_1}Q_{e_2}  (e^{t X_{e_3}}Q_{e_3})^* Q_{e_4}^*
= Q_{e_1}Q_{e_2}Q_{e_3}^* X_{e_3}^* Q_{e_4}^*
\end{equ}
which is effectively just inserting the matrix $X_{e_3}^*$.

Note that:

(1) For the terms with $e=\bar e$, the term involving plaquette $p$ in $\cS$ 
	is non-zero if and only if
the plaquette $p$  contains $e$ or $e^{-1}$. In this case, we write $p\in \cP_e$, and there will be $2(d-1)$ such plaquettes $p$.
Direct calculation as in \eqref{e:X3Q} yields a result of the following form
\begin{equ}
 ( X_{e}Q_{e})(X_eQ_e)\tr (Q_p ) =
 \tr( Y_1 Q_{e_1} Y_2 Q_{e_2}Q_{e_3}^* Y_3^* Q_{e_4}^* Y_4^*)
\end{equ}
for some matrices $Y_1,Y_2, Y_3, Y_4$,
in which  three of them are $I_N$ and one of them is $X_e^2$.

(2) For the terms with $e\neq \bar e$, the summand on the RHS of \eqref{e:XQXQ} is non-zero if and only if
there exists a plaquette $p$ which contains both $e$ or $e^{-1}$ and $\bar e$ or $\bar e^{-1}$.
In this case, we write $p\in \cP_{e,\bar e}$ and there will be only one such plaquette, and
\begin{equ}
 ( X_{\bar e}Q_{\bar e})(X_eQ_e)\tr (Q_p ) =
 \tr( Y_1 Q_{e_1} Y_2 Q_{e_2}Q_{e_3}^* Y_3^* Q_{e_4}^* Y_4^*)
\end{equ}
where two of $Y_1,Y_2, Y_3, Y_4$ are $I_N$, and the other two are $X_e$ and $X_{\bar e}$.

In either case, there are two occurrences of $X$, so by cyclic invariance of trace we can write the result into one of the following forms:
$$
\tr (Q X_{e} \cdot \tilde Q X_{\bar e}),
\qquad
\tr (Q X_{e} \cdot (\tilde Q X_{\bar e})^*),
\qquad
\tr ( (Q X_{e})^* \cdot \tilde Q X_{\bar e}),
\qquad
\tr ((Q X_{e})^* \cdot (\tilde Q X_{\bar e})^*)
$$
for some $Q,\tilde Q\in G$. By the Cauchy--Schwarz inequality for the Hilbert--Schmidt inner product,
each of these terms is bounded by
$$
\Big(\tr (Q X_{e} \cdot (Q X_{e})^*)\Big)^{\frac12} \; \Big(\tr (\tilde Q X_{\bar e} \cdot (\tilde Q X_{\bar e})^*)\Big)^{\frac12}
= |X_e| |X_{\bar e}|
$$
and when  $e\neq \bar e$ this is bounded by $\frac12 (|X_e|^2+|X_{\bar e}|^2)$.

Therefore we have
\begin{equ}
\frac{1}{N}\sum_{e=\bar e\in E_\Lambda^+ } |( X_{\bar e}Q_{\bar e})(X_eQ_e)\cS|
\le
 \sum_{e \in E_\Lambda^+ } \sum_{p \in \cP_e} |\beta||X_e|^2
=
 \sum_{e \in E_\Lambda^+ } 2|\beta|(d-1) |X_e|^2
 =
  2|\beta|(d-1) |v|^2
\end{equ}
and
\begin{equs}
\frac{1}{N}&
\sum_{e \neq \bar e\in E_\Lambda^+ }  | ( X_{\bar e}Q_{\bar e})(X_eQ_e)\cS|
\le
 \sum_{e \neq \bar e\in E_\Lambda^+ } \sum_{p \in \cP_{e,\bar e}}\frac{|\beta|}2 (|X_e|^2+|X_{\bar e}|^2)
 \\
&=
\sum_{p \in \cP^+_\Lambda} \sum_{e \neq \bar e \in E_\Lambda^+} 1_{p \in \cP_{e,\bar e}}
\cdot \frac{|\beta|}2 (|X_e|^2+|X_{\bar e}|^2)
=
\sum_{p \in \cP^+_\Lambda} \sum_{e  \in E_\Lambda^+} 1_{p \in \cP_{e}}
\cdot 3 |\beta||X_e|^2
\\
&= \sum_{e  \in E_\Lambda^+}   \sum_{p \in \cP_{e}}
3|\beta| |X_e|^2
= \sum_{e  \in E_\Lambda^+}
6(d-1)|\beta| |X_e|^2
 =
6(d-1) |\beta||v|^2
\end{equs}
which implies \eqref{eq:hess}.
\end{proof}

We denote the Riemannian distance on $G$  by $\rho$.
We write $\rho_L$ for the induced Riemannian distance  on $\cQ_L$ given by
\begin{align*}
	\rho_L(Q,Q')^2\eqdef \sum_{e\in E_{\Lambda_{L}}^+}\rho(Q_e,Q'_e)^2, \quad Q,Q'\in \cQ_L.
\end{align*}
For any $\mu, \nu\in \sP(\cQ_L)$, we introduce the Wasserstein distance as
\begin{align*}
	W_p^{\rho_L}(\mu,\nu)\eqdef \inf_{\pi\in \sC(\mu,\nu)}\pi(\rho_L^p)^{1/p},
\end{align*}
with $\sC(\mu,\nu)$ being the set of couplings between $\mu$ and $\nu$.

We then have the following result using the Bakry--\'Emery condition \eqref{c1} and \cite[Theorem 5.6.1]{Wang}, which was first proved by
\cite{Bakry97} and
\cite{RS05}.

\bt\label{th:4.2} 
Under Assumption \ref{ass1}, the following hold.
\begin{enumerate}
\item The dynamic defined by the SDE \eqref{eq:YM} is exponentially ergodic in the sense that
\begin{align}\label{eq:er}
	W_2^{\rho_L}(\delta_QP_t^L,\delta_{\bar Q}P_t^L)\leq e^{-K_\cS t}\rho_L(Q,\bar{Q}), \quad t\geq 0, \quad  Q,\bar Q\in \cQ_L.
\end{align}
\item  For $1<p<2$
\begin{align}\label{eq:er1}
	W_p^{\rho_L}(\mu P_t^L,\nu P_t^L)\leq e^{-K_\cS t}W_p^{\rho_L}(\mu ,\nu ), \quad t\geq 0, \quad\mu,\nu\in \sP(\cQ_L),
\end{align}
\end{enumerate}
In particular,  invariant measure of $(P_t^L)_{t\geq0}$ is unique.
 \et


\begin{proof}
Using \cite[Theorem 5.6.1(1)(11)(12)]{Wang}  we know that
 \eqref{eq:er},  \eqref{eq:er1} and
the following condition are
 all equivalent:  for every $v=XQ\in T_Q \cQ_L$,
\begin{align}\label{c1}
	\Ric(v,v)- \<\nabla_v\nabla \cS,v\>\geq K_\cS|X|^2.
\end{align}
Here we recall that $|v|^2=|X|^2$
and $\<\nabla_v\nabla \cS,v\>=\Hess_{\cS}(v,v)$.
By \cite[(F.6)]{AGZ},  for any tangent vector $u$ of $G$,
\begin{align*}
	\Ric(u,u)= \Big(\frac{\alpha(N+2)}4-1\Big)|u|^2,
\end{align*}
with $\alpha=1, 2$ for $SO(N)$ and $SU(N)$ 
respectively.
Since $\Ric(v,v)=\sum_{e}\Ric(v_e,v_e)$
and $|X|^2=\sum_{e} |X_e|^2$,
 we have
\begin{equ}[e:Ricvv]
	\Ric(v,v)= \Big(\frac{\alpha(N+2)}4-1\Big)|X|^2.
\end{equ}
By Lemma~\ref{lem:4.1} and
definition of $K_\cS$ in Assumption \ref{ass1},
we obtain  \eqref{c1}, and therefore
\eqref{eq:er},  \eqref{eq:er1} follow.

Uniqueness of invariant measure follows from \eqref{eq:er1} by letting $t\to \infty$.
%
%
\end{proof}

 \br  In general, if we do not require $K_\cS $ to be strictly positive as in
 Assumption \ref{ass1}, \eqref{eq:er}-\eqref{eq:er1} still hold,
 and \eqref{c1} is also equivalent with the following statements: for any $ t\geq0, f\in C^1(\cQ_L)$
\begin{equs}\label{eq:nap}
	|\nabla P_t^Lf|  &  \leq e^{-K_\cS t}P_t^L|\nabla f|,
\\
	P_t^L(f^2\log f^2)-(P_t^Lf^2)\log (P_t^Lf^2) & \leq \frac{2(1-e^{-2K_\cS t})}{K_\cS } P_t^L|\nabla f|^2.\label{eq:pt}
\end{equs}
We refer to \cite[Theorem 5.6.1]{Wang} for these results and more equivalent statements.
\er


As \eqref{c1} is the Bakry--\'Emery's condition, we have the following log-Sobolev inequality (c.f. \cite[Theorem 5.6.2]{Wang}). In fact, it follows from taking integral w.r.t. $\mu_{\Lambda_L,N,\beta}$ on the both sides of \eqref{eq:pt} and letting $t\to\infty$.

\bc Under Assumption \ref{ass1}, the log-Sobolev inequality holds for each $L>1$, i.e.
for $F\in C^\infty(\cQ_L)$ with $ \mu_{\Lambda_L,N,\beta}(F^2)=1$,
\begin{align*}
	\mu_{\Lambda_L,N,\beta}(F^2\log F^2)\leq \frac2{K_\cS }\cE^L(F,F).
\end{align*}
This implies the Poincar\'e inequality: for $F\in C^\infty(\cQ_L)$,
\begin{align}\label{eq:p1}
	\mu_{\Lambda_L,N,\beta}(F^2)\leq \frac1{K_\cS }\cE^L(F,F)+\mu_{\Lambda_L,N,\beta}(F)^2.
\end{align}
\ec

We could view any probability measure $\nu$ in $\sP(\cQ_L)$ as a probability measure  in $\sP(\cQ)$ by periodic extension. Namely, we can construct a random variable with law given by $\nu \in \sP(\cQ_L)$ and extend the random variable periodically. The law of the periodic extension gives the desired extension of $\nu$.
 Since $G$ is compact,  $\{\mu_{\Lambda_L,N,\beta}\}_{L}$ form a tight set and passing to a subsequence we obtain a tight limit, which is denoted by $\mu_{N,\beta}$.
 Hence, by approximation we have the following results.

 \bc\label{co:lo}
 Under Assumption \ref{ass1},
 the log-Sobolev inequality holds, i.e. for cylinder functions $F\in C^\infty_{cyl}(\cQ)$ with $\mu_{N,\beta}(F^2)=1$,
 \begin{equ}[e:log]
 	\mu_{N,\beta}(F^2\log F^2)\leq \frac2{K_\cS }\cE^{\mu_{N,\beta}}(F,F).
 \end{equ}
This implies the Poincar\'e inequality: for cylinder functions $F\in C^\infty_{cyl}(\cQ)$
 \begin{equ}[e:PoinE]
 	\mu_{N,\beta}(F^2)\leq \frac1{K_\cS }\cE^{\mu_{N,\beta}}(F,F)  +\mu_{N,\beta}(F)^2.
 \end{equ}
 \ec

In Section~\ref{sec:uniqueYM} we will prove Theorem \ref{th:1.3} which will then identify the tight limit
$\mu_{N,\beta}$ in \eqref{e:log} and \eqref{e:PoinE} as the measure $\muYM$ in Theorem \ref{th:1.3}; this then proves Theorem~\ref{th:1}.

\br\label{re1} By the Poincar\'e inequality \eqref{eq:p1} and \eqref{e:PoinE}, the semigroup $(P_t^L)_{t\geq0}$ and  $(P_t)_{t\geq0}$ satisfy
\begin{align*}
\|	P_t^Lf-\mu_{\Lambda_L,N,\beta}(f)\|_{L^2(\mu_{\Lambda_L,N,\beta})}\leq e^{-tK_\cS }\|f\|_{L^2(\mu_{\Lambda_L,N,\beta})},
\end{align*}
and 
\begin{align*}
	\|	P_tf-\mu_{N,\beta}(f)\|_{L^2(\mu_{N,\beta})}\leq e^{-tK_\cS }\|f\|_{L^2(\mu_{N,\beta})}.
\end{align*}
(c.f. \cite[Theorem 1.1.1]{Wang}). However, this does not imply the uniqueness of the invariant measure for $(P_t)_{t\geq0}$.
\er

The following two results are simple applications of the Poincar\'e inequality. 

\bc\label{co:c}
Under Assumption \ref{ass1},
for every  
$e_0\in E^+$ and every unit vector $E$ in $M_N$ 
we have
\begin{align*}
	\sum_{e\in E_{\Lambda_L}^+} \cov_{N,\beta,L}\Big(\<Q_{e_0},E\>,\<Q_e,E\>\Big)\leq  \begin{cases}
		1 / K_\cS ,&\quad G= SO(N),\\
		2/ K_\cS ,&\quad G=SU(N).
	\end{cases}
\end{align*}
Here $\cov_{N,\beta,L}$ means covariance w.r.t. the measure $\mu_{\Lambda_L,N,\beta}$.
In particular,
\begin{align*}
	\Big|\sum_{e\in E_{\Lambda_L}^+ \backslash\{e_0\}} \cov_{N,\beta,L}\Big(\<Q_{e_0},E\>,\<Q_e,E\>\Big)\Big|\leq \begin{cases}
		2 / K_\cS ,&\quad G= SO(N),\\
		4 / K_\cS ,&\quad G=SU(N).
	\end{cases}
\end{align*}
\ec



\begin{proof}
	Let $f= |E_{\Lambda_L}^+|^{-\frac12}\sum_{e\in E_{\Lambda_L}^+} \<Q_e,E\>$. By direct calculation, one has $\nabla \<Q_e,E\>=\proj(EQ_e^*)Q_e$, which implies that
	$|\nabla f|^2\leq \gamma$ with $\gamma=1$ for $G=SO(N)$ and $\gamma=2$ for $G=SU(N)$, where for $G=SU(N)$ we used that for any matrices 
	$Q, Q'\in M_N$
	\begin{align}\label{eq:tr}
		\tr\Big((Q-\frac1N\tr(Q)I_N)(Q'-\frac1N\tr(Q')I_N)\Big)
		=\tr(QQ')-\frac1N\tr(Q)\tr(Q').
	\end{align}
	Hence, by the Poincar\'e inequality \eqref{eq:p1} we get
	\begin{align*}
		\frac1{|E_{\Lambda_L}^+|}\sum_{e,e'\in E_{\Lambda_L}^+}\cov_{N,\beta,L}\Big(\<Q_e,E\>,\<Q_{e'},E\>\Big)\leq \frac\gamma{K_\cS }.
	\end{align*}
	With periodic boundary condition we have translation invariance,
	so for fixed edge $e_0$
	\begin{align*}
		\sum_{ e\in E_{\Lambda_L}^+ } \!\! \cov_{N,\beta,L}\Big(\<Q_{e_0},E\>,\<Q_e,E\>\Big)\leq \frac\gamma{K_\cS },
	\end{align*}
	which implies the first result and
	\begin{align*}
		\Big|  \sum_{ e\in E_{\Lambda_L}^+ \backslash\{e_0\}} \!\!\!\!  \cov_{N,\beta,L}\Big(\<Q_{e_0},E\>,\<Q_e,E\>\Big)\Big|\leq \frac\gamma{K_\cS }+\var_{N,\beta,L}\Big(\<Q_{e_0},E\>\Big),
	\end{align*}
	where we used triangle inequality and  $\var_{N,\beta,L}$ means variance under $\mu_{\Lambda_L,N,\beta}$.
	Now we take $g=\<Q_{e_0},E\>$ and have $|\nabla g|^2\leq \gamma$. Then by the Poincar\'e inequality \eqref{eq:p1}
	$$\var_{\Lambda_L,N,\beta}\Big(\<Q_{e_0},E\>\Big)\leq \frac\gamma{K_\cS }.$$
	Thus the second result follows.
\end{proof}

\bc \label{co:cc}
Under Assumption \ref{ass1}, it holds that for  every plaquette $p$ in $\mathcal{P}^+_{\Lambda_{L}}$
\begin{align*}
	\sum_{\bar p\in \mathcal{P}_{\Lambda_L}} \cov_{N,\beta,L}\Big(\mathrm{ReTr}Q_p,\mathrm{ReTr} Q_{\bar p}\Big)\leq \begin{cases}
		8N(d-1) / K_\cS ,&\quad G= SO(N),\\
		16N(d-1) / K_\cS ,&\quad G=SU(N).
	\end{cases}
\end{align*}
Here $\cov_{N,\beta,L}$ means covariance w.r.t. the measure $\mu_{\Lambda_L,N,\beta}$. In particular,
for  every plaquette $p$ in $\mathcal{P}^+_{\Lambda_{L}}$
\begin{align*}
	\Big|\sum_{\bar p\in \mathcal{P}_{\Lambda_L},\bar p\neq p} \cov_{N,\beta,L}\Big(\mathrm{ReTr}Q_p,\mathrm{ReTr} Q_{\bar p}\Big)\Big|\leq \begin{cases}
		\big(8N(d-1)+4N\big) / K_\cS ,&\quad G= SO(N),\\
		\big(16N(d-1)+8N \big)/ K_\cS,&\quad G=SU(N).
	\end{cases}
\end{align*}
\ec
\begin{proof}
	Let $f= |\mathcal{P}_{\Lambda_L}^+|^{-\frac12}     \sum_{\bar p\in \mathcal{P}_{\Lambda_L}^+}\mathrm{ReTr} Q_{\bar p}$.
	By the same calculation as \eqref{e:grad} and \eqref{e:DS} for the action $\cS$ we have
	\begin{align*}
		\nabla f_{e}=
		\begin{cases}
			\displaystyle
			-\frac1{2|\mathcal{P}_{\Lambda_L}^+|^{1/2}} \sum_{p\in \cP_\Lambda,p\succ e}(Q_p-Q_p^*)Q_e\;,
			&\qquad
			G= SO(N) \;,
			\\
			\displaystyle
			-\frac1{2|\mathcal{P}_{\Lambda_L}^+|^{1/2}}\sum_{p\in \cP_\Lambda ,p\succ e}
			\Big( (Q_p-{Q}_p^{*}) - \frac{1}{N}\tr(Q_p-{Q}_p^{*}) I_N\Big)   Q_e\;,
			&\qquad
			G = SU(N)\;.
		\end{cases}
	\end{align*}
	Thus in $SO(N)$ case we have
	\begin{align*}
		|\nabla f_e|^2
		=\frac1{4|\mathcal{P}_{\Lambda_L}^+|}\sum_{p,\bar p\in \cP_\Lambda,p,\bar p\succ e}\mathrm{Tr}\Big((Q_p-Q_p^*)(Q_{\bar p}-Q_{\bar p}^*)^*\Big)\leq \frac{4N(d-1)^2}{|\mathcal{P}_{\Lambda_L}^+|},
	\end{align*}
	which implies that
	\begin{align*}|\nabla f|^2\leq 8N(d-1).
	\end{align*}
	Here we used
	$|\mathcal{P}_{\Lambda_L}^+|
	=\frac{2(d-1)}{4}|E_{\Lambda_L}^+|
	$, since each plaquette has $4$ edges and each edge is adjacent to $2(d-1)$ plaquettes.
	Hence,  applying the Poincar\'e inequality to $f$ we get
	\begin{align*}
		\frac1{|\mathcal{P}_{\Lambda_L}^+|}\sum_{p,\bar p\in \mathcal{P}_{\Lambda_L}^+}\cov_{N,\beta,L}\Big(\mathrm{ReTr}Q_p,\mathrm{ReTr} Q_{\bar p}\Big)\leq \frac{8N(d-1)}{K_\cS }.
	\end{align*}
	We  choose periodic boundary condition to have translation invariance and we get for fixed  $p$
	\begin{align*}
		\sum_{\bar p\in \mathcal{P}_{\Lambda_L}^+}\cov_{N,\beta,L}\Big(\mathrm{ReTr}Q_p,\mathrm{ReTr} Q_{\bar p}\Big)\leq \frac{8N(d-1)}{K_\cS }.
	\end{align*}
	Thus the first result follows and
	\begin{align*}
		\Big|
		\sum_{\bar p\in \mathcal{P}_{\Lambda_L}^+,\bar p\neq p}\cov_{N,\beta,L}\Big(\mathrm{ReTr}Q_p,\mathrm{ReTr} Q_{\bar p}\Big)
		\Big|
		\leq \frac{8N(d-1)}{K_\cS }+\var\Big(\mathrm{ReTr}Q_p \Big) .
	\end{align*}
	Moreover, take $f=\mathrm{ReTr} Q_{\bar p}$  and $|\nabla f|^2\leq 4N$ 
	and we obtain
	$$
	\var\Big(\mathrm{ReTr}Q_p\Big)\leq \frac{4N}{K_\cS }.
	$$
	Thus the second result follows for $SO(N)$ case.
	The result for the $SU(N)$ case follows by similar arguments and using \eqref{eq:tr}.
\end{proof}

\subsection{Application I: large $N$ limit of Wilson loops}

In the following we give the proof of Corollary \ref{co:1} by applying the Poincar\'e inequality.

\begin{proof}[Proof of Corollary \ref{co:1}]
Since Theorem \ref{th:1.3}   identifies any tight limit
$\mu_{N,\beta}$ as the measure $\muYM$, 
it suffices to prove the result for any tight limit $\mu_{N,\beta}$.
	We apply the Poincar\'e inequality \eqref{e:PoinE} to Wilson loops defined in \eqref{e:loop}.
	Consider the $SO(N)$ case.
	Let
	\begin{align*}
		f(Q)=\frac1NW_\ell=\frac1N\tr(Q_{e_1}Q_{e_2}\dots Q_{e_n}).
	\end{align*}
	We get
	\begin{align*}
		\mu_{N,\beta}(f^2)-\mu_{N,\beta}(f)^2=\var\Big(\frac1NW_\ell\Big).
	\end{align*}
	We then need to calculate $\nabla f$ which appears on the RHS of the Poincar\'e inequality.
	For an edge which appears in the location $x$ of the loop $\ell$, we write
	$$
	Q_\ell = \prod_{i=1}^n  Q_{e_i},
	\qquad
	Q_{a_x}=\prod_{i=1}^{x-1}Q_{e_i},
	\qquad
	Q_{b_x}=\prod_{i=x+1}^nQ_{e_i}.
	$$
	We then have $W_\ell=\tr(Q_\ell)$.
	For each $e\in E^+$, we may have an edge $e_x$ in $\ell$ which is $e$ or $e^{-1}$, so by straightforward calculation we have
	\begin{equation}\label{eq:Well}
		\aligned
		(\nabla W_\ell)_{e}=&-\frac12\sum_{x=1}^n\1_{e_x=e}(Q_{e_x}Q_{b_x}Q_{a_x}-Q_{a_x}^*Q_{b_x}^*Q_{e_x}^*)Q_{e}
		\\&+\frac12\sum_{x=1}^n\1_{e_x=e^{-1}}(Q_{b_x}Q_{a_x}Q_{e}^*-Q_{e}Q_{a_x}^*Q_{b_x}^*) Q_{e}.
		\endaligned
	\end{equation}
	Here, the calculation is similar as in Lemma~\ref{lem:gradS} (see \cite{SSZloop}).
	Namely, when $e_x=e$, by cyclic invariance of trace,
	we write $W_\ell=\tr(Q_{e_x}Q_{b_x}Q_{a_x})$, and for $X\in \mfg$
	we compute
	\begin{equs}[e:dW1]
		\partial_t |_{t=0}  \tr(e^{tX}Q_{e_x}Q_{b_x}Q_{a_x})
		&=\tr(X Q_{e_x}Q_{b_x}Q_{a_x})
		= \< X, \proj (Q_{e_x}Q_{b_x}Q_{a_x})^*\>
		\\
		&= \< X Q_{e_x}, \proj (Q_{e_x}Q_{b_x}Q_{a_x})^*  Q_{e_x}\>
	\end{equs}
	where $\proj$ is defined in  Lemma~\ref{lem:gradS}  and is the orthogonal projection of a matrix to the Lie algebra $\mfg$ of skew-symmetric matrices, and $X Q_{e_x}$ is a tangent vector at $ Q_{e_x}$.
On the other hand if $e_x=e^{-1}$, we write $W_\ell=\tr(Q_{b_x}Q_{a_x}Q_{e}^*)$, so
 for $X\in \mfg$ similar calculation as above yields
	\begin{equ}[e:dW2]
	\tr(Q_{b_x}Q_{a_x}Q_{e}^* X^*)
		= \< X, \proj (Q_{b_x}Q_{a_x}Q_{e}^* )\>
		= \< X Q_{e}, \proj (Q_{b_x}Q_{a_x}Q_{e}^* ) Q_{e}\>
	\end{equ}
which gives the second term on the RHS of \eqref{eq:Well}.

	Using \eqref{eq:Well} we have
	\begin{align*}
		|(\nabla  W_\ell)_{e}|^2
		&=\frac14\sum_{x,y=1}^n\1_{e_x=e_y=e}
		\tr\Big((Q_{e_x}Q_{b_x}Q_{a_x}Q_{e_x}-Q_{a_x}^*Q_{b_x}^*)(Q_{e_y}Q_{b_y}Q_{a_y}Q_{e_y}-Q_{a_y}^*Q_{b_y}^*)^*\Big)
		\\
		&-\frac14\sum_{x,y=1}^n\1_{e_x^{-1}=e_y=e}
		\tr\Big((Q_{b_x}Q_{a_x}-Q_{e}Q_{a_x}^*Q_{b_x}^*Q_{e})(Q_{e_y}Q_{b_y}Q_{a_y}Q_{e_y}-Q_{a_y}^*Q_{b_y}^*)^*\Big)
		\\
		&-\frac14\sum_{x,y=1}^n\1_{e_x=e_y^{-1}=e}
		\tr\Big((Q_{e_x}Q_{b_x}Q_{a_x}Q_{e_x}-Q_{a_x}^*Q_{b_x}^*)(Q_{b_y}Q_{a_y}-Q_{e}Q_{a_y}^*Q_{b_y}^*Q_{e})^*\Big)
		\\
		&+\frac14\sum_{x,y=1}^n\1_{e_x=e_y=e^{-1}}
		\tr\Big((Q_{b_x}Q_{a_x}-Q_{e}Q_{a_x}^*Q_{b_x}^*Q_{e})(Q_{b_y}Q_{a_y}-Q_{e}Q_{a_y}^*Q_{b_y}^*Q_{e})^*\Big).
	\end{align*}
	Note that the trace of any $SO(N)$ matrix is bounded by $N$, and therefore each of the four traces above is bounded by $4N$.
	Summing over $e\in E^+$, we see that the Dirichlet form
	term in the Poincar\'e inequality is bounded as follows:
	\begin{equs}\label{eq:ef}
		\cE^{ \mu_{N,\beta}}(f,f) &= \sum_{e\in E^+}  \mu_{N,\beta}   (|\nabla_{e} f|^2)  \\
		&\leq
		\frac{1}{N}\sum_{e\in E^+}\sum_{x,y=1}^n
		\Big(\1_{e_x=e_y=e}
		+\1_{e_x^{-1} =e_y=e}
		+\1_{e_x =e_y^{-1}=e}
		+\1_{e_x=e_y=e^{-1}}\Big).
	\end{equs}
	For any  edge $e\in E^+$ we let $A(e)$ be the number of locations in $\ell$ where $e$ occurs and  $B(e)$ be the number of locations in $\ell$ where $e^{-1}$ occurs.
	\eqref{eq:ef} is then bounded by
	\begin{align*}
		\frac{1}{N}\sum_{e\in E^+}(A(e)+B(e))^2\leq \frac{n(n-3)}{N},
	\end{align*}
	where we used $\sum_{e\in E^+}(A(e)+B(e))=n$ and $A(e)+B(e)\leq n-3$.
	The Poincar\'e inequality then yields
	\begin{align*}
		\var\Big(\frac1NW_\ell\Big)\leq \frac1{K_\cS }\frac{n(n-3)}{N}.
	\end{align*}
	Letting $N\to\infty$, \eqref{e:EW} follows for the $SO(N)$ case.
	
	For $G=SU(N)$ we  choose, with $\iota=\sqrt{-1}$,
	\begin{equs}
		f_{R}(Q) & =\frac1N\mathrm{Re}W_\ell=\frac1N\mathrm{Re}\tr\Big(Q_{e_1}Q_{e_2}\dots Q_{e_n}\Big),
		\\
		f_{I}(Q) & =\frac1N\mathrm{Im}W_\ell=-\frac1N\mathrm{Re}\tr\Big(\iota Q_{e_1}Q_{e_2}\dots Q_{e_n}\Big)
	\end{equs}
	to obtain the result for the real and imaginary  parts. It is sufficient to calculate $\nabla\mathrm{Re} W_\ell$. Besides the terms in \eqref{eq:Well} we also have the following additional terms
	\begin{align*}
		&\frac12\sum_{x=1}^n\1_{e_x=e}\frac1N\tr(Q_{e_x}Q_{b_x}Q_{a_x}-Q_{a_x}^*Q_{b_x}^*Q_{e_x}^*)Q_{e}
		\\&-\frac12\sum_{x=1}^n\1_{e_x=e^{-1}}\frac1N\tr(Q_{b_x}Q_{a_x}Q_{e_x}-Q_{e}Q_{a_x}^*Q_{b_x}^*)Q_{e}.
	\end{align*}
	(This is similar with how the second case in \eqref{e:DS} was derived, namely, the projection $\proj$ appearing in \eqref{e:dW1}\eqref{e:dW2} should also make the matrices traceless in the $SU(N)$ case.)
	Noting that \eqref{eq:tr} 
	we have 
	\begin{align*}
		|(\nabla  \mathrm{Re}W_\ell)_{e}|^2
		\leq
		2\sum_{x,y=1}^n
		\Big(\1_{e_x=e_y=e}
		+\1_{e_x^{-1} =e_y=e}
		+\1_{e_x =e_y^{-1}=e}
		+\1_{e_x=e_y=e^{-1}}\Big)N.
	\end{align*}
	Summing over $e\in E^+$ we get
	\begin{align*}
		\cE(f_R,f_R)
		&\leq
		\frac{2}{N}\sum_{e\in E^+}\sum_{x,y=1}^n
		\Big(\1_{e_x=e_y=e}
		+\1_{e_x^{-1} =e_y=e}
		+\1_{e_x =e_y^{-1}=e}
		+\1_{e_x=e_y=e^{-1}}\Big)
		\\ &
		\leq\frac{2n(n-3)}N.
	\end{align*}
	Similarly, we get
	\begin{align*}
		\cE(f_I,f_I)\leq\frac{2n(n-3)}N.
	\end{align*}
	Hence, \eqref{e:EW} holds for $SU(N)$.
	
	To prove the factorization property,
	by the Cauchy--Schwarz inequality we have
	\begin{align*}
		&N^{-n}
		\Big|\E(W_{\ell_1}\dots W_{\ell_n})
		-\E(W_{\ell_1}\dots W_{\ell_{n-1}})\E W_{\ell_n}\Big|
		\\
		&\leq N^{-n}
		\Big|\E(W_{\ell_1}\dots W_{\ell_{n-1}}(W_{\ell_n}-\E W_{\ell_n})\Big|
		\\&
		\leq \var\Big(\frac{W_{\ell_n}}N\Big)^{1/2}\to0.
	\end{align*}
	Hence, the result follows by induction.
\end{proof}



\subsection{Application II: mass gap}

In this section we use the Poincar\'e inequality to prove the existence of mass gap for  lattice Yang--Mills. To this end, 
for $f\in C^\infty_{cyl}(\cQ)$,  recall that  $\Lambda_f$ is the set of edges $f$ depends on. We define
\begin{align*}
	\$f\$_{\infty}\eqdef \sum_{e\in \Lambda_f}\|\nabla_{e}f\|_{L^\infty}.
\end{align*}
In this section it will be convenient for the calculations to consider an explicit choice of an orthonormal basis
of $\mfg$. This choice is standard, see e.g. \cite[Proposition E.15]{AGZ}.
Let $e_{kn}\in M_N$ for $k,n=1,\dots, N$ be the elementary matrices, namely its $(k,n)$-th entry is $1$ and all the other entries are $0$.
For $1\leq k< N$ and $\iota=\sqrt{-1}$, let
\begin{align*}
	D_k=\frac{\iota}{\sqrt{k+k^2}}\Big(-ke_{k+1,k+1}+\sum_{i=1}^ke_{ii}\Big).
\end{align*}
For $1\leq k,n\leq N$, let 
\begin{equ}[e:def-EF]
	E_{kn}=\frac{e_{kn}- e_{nk}}{\sqrt 2},\quad
	F_{kn}=\frac{\iota e_{kn} + \iota e_{nk}}{\sqrt 2}.
\end{equ}
Then:
\begin{itemize}
\item
$\{E_{kn}:1\leq k<n\leq N\}$ is an orthonormal basis of $\so(N)$, and,
\item
$\{D_k:1\leq k< N\}\cup\{E_{kn},F_{kn}:1\leq k<n\leq N\}$ is an orthonormal basis of $\su(N)$.
\end{itemize}
This then determines an orthonormal basis $\{v_e^i\}$ of $\mfg^{E^+_{\Lambda_L}}$, which consists of right-invariant vector fields on $\cQ_L$.

We first prove the following lemma for Lie brackets.

\bl\label{lem:com} It holds that for every $v_e^i$ 
$$
\sum_j |[v_e^i,v_e^j]f|^2 \leq \frac12|\nabla_ef|^2 \quad \mbox{for} \quad G=SO(N),
$$
$$
\sum_j |[v_e^i,v_e^j]f|^2 \leq \frac{9}{2}|\nabla_ef|^2 \quad \mbox{for} \quad G=SU(N).
$$
\el 
\begin{proof}
By direct calculation we have 
$$
e_{ij} e_{mn} = \delta_{jm} e_{in}.
$$
Using this and \eqref{e:def-EF}, we deduce
\begin{equs}
2[E_{kn} , E_{lm}] &= [e_{kn} - e_{nk}, e_{lm}-e_{ml}]
\\
&=\delta_{nl} e_{km} - \delta_{km} e_{ln}
-\delta_{nm} e_{kl} + \delta_{lk} e_{mn}
\\
&\quad  -\delta_{kl} e_{nm}+\delta_{mn} e_{lk}
+\delta_{km} e_{nl}-\delta_{nl} e_{mk}\label{eq:s1}
\\
&=
\delta_{nl} (e_{km}-e_{mk})
+\delta_{km}  (e_{nl}-e_{ln})
+\delta_{nm} (e_{lk}-e_{kl})
+\delta_{lk}  (e_{mn}-e_{nm}).
\end{equs}
 With this calculation, observe that if we fix $(k,n)$ and vary $(l,m)$,
we either get $0$ or one of the orthonormal basis vectors of $\so(N)$
up to a factor $\pm \frac1{\sqrt 2}$,
and in the latter case different values of $(l,m)$ yield different basis vectors.
This implies that  for $G=SO(N)$,  $\{[v_e^i,v_e^j],j=1,\dots,\dg\}$ is a subset of orthonormal basis of $T_{Q_e}G$ up to a factor $\pm \frac1{\sqrt 2}$. Hence, the result holds for $SO(N)$ by definition of $|\nabla_ef|^2$.

The proof for the $SU(N)$ case is similar but requires a bit more calculations.
We have
\begin{equs}
\frac{2}{\iota}[E_{kn} , F_{lm}] &= [e_{kn} - e_{nk}, e_{lm}+e_{ml}]
\\
&=\delta_{nl} e_{km} - \delta_{km} e_{ln}
+\delta_{nm} e_{kl} - \delta_{lk} e_{mn}
\\
&\quad  -\delta_{kl} e_{nm}+\delta_{mn} e_{lk}
-\delta_{km} e_{nl} +\delta_{nl} e_{mk}\label{eq:s2}
\\
&=
\delta_{nl} (e_{km}+e_{mk})
+\delta_{km}  (-e_{nl}-e_{ln})
+\delta_{nm} (e_{lk}+e_{kl})
+\delta_{lk}  (-e_{mn}-e_{nm}).
\end{equs}
When $k=l\neq n=m$
\begin{equs}
\frac{2}{\iota}[E_{kn} , F_{lm}] 
&=
\delta_{nm} (e_{lk}+e_{kl})
+\delta_{lk}  (-e_{mn}-e_{nm})
= 2e_{kk} - 2e_{nn}.\label{eq:s}
\end{equs}
Furthermore
\begin{equs}
-2[F_{kn} , F_{lm}] &= [e_{kn} + e_{nk}, e_{lm} + e_{ml}]
\\
&=\delta_{nl} e_{km} - \delta_{km} e_{ln}
+\delta_{nm} e_{kl} - \delta_{lk} e_{mn}
\\
&\quad  +\delta_{kl} e_{nm}-\delta_{mn} e_{lk}
+\delta_{km} e_{nl}-\delta_{nl} e_{mk}\label{eq:s3}
\\
&=
\delta_{nl} (e_{km}-e_{mk})
+\delta_{km}  (e_{nl}-e_{ln})
+\delta_{nm} (e_{kl}-e_{lk})
+\delta_{lk}  (e_{nm}-e_{mn}).
\end{equs}
 For Lie brackets involving $D$   we  have
\begin{align*}
[E_{kn}, e_{mm}] &= ( \delta_{mn} F_{mk} - \delta_{mk} F_{mn} )/\iota.
\\
 [F_{kn}, e_{mm}] 
&= \iota ( \delta_{mn} E_{km} + \delta_{mk} E_{nm} ).
	\end{align*}
By this we  obtain, for  $k<n$,
\begin{equation}\label{eq:s4}[E_{kn},D_m]=
	\begin{cases}
		0& m+1<k\\
		\frac{m}{\sqrt{m+m^2}}F_{kn}&m+1=k\\
		-\frac1{\sqrt{m+m^2}}F_{kn}&k\leq m<m+1<n\\
		-\frac{m+1}{\sqrt{m+m^2}}F_{kn}& k\leq m<m+1=n\\
		0&k,n\leq m,
	\end{cases}
\end{equation}
and 
\begin{equation}\label{eq:s5}
	[F_{kn},D_m]=
	\begin{cases}
		0& m+1<k\\
		-\frac{m}{\sqrt{m+m^2}}E_{kn}&m+1=k\\
		\frac1{\sqrt{m+m^2}}E_{kn}&k\leq m<m+1<n\\
		\frac{m+1}{\sqrt{m+m^2}}E_{kn}& k\leq m<m+1=n\\
		0&k,n\leq m.
	\end{cases}
\end{equation}

For $v_e^i= E_{kn}Q_e$ or $F_{kn}Q_e$,  we decompose $\{v_e^p,p=1,\dots, \dg\}$ into the following three sets
$$I_1=\{E_{lm}Q_e, F_{lm}Q_e,1\leq l<m\leq N,l\neq k \text{ or } n\neq m\},\quad I_2=\{E_{kn}Q_e,F_{kn}Q_e\},$$
and 
$$I_3=\{D_kQ_e,1\leq k\leq N-1\}.$$
By \eqref{eq:s1}-\eqref{eq:s3} we view $\{[v_e^i,v_e^j], v_e^j\in I_1\}$ as 
a subset of orthonormal basis  of $T_{Q_e}G$ up to a factor $\pm \frac1{\sqrt 2}$. Hence, 
$$
\sum_{v_e^j\in I_1} |[v_e^i,v_e^j]f|^2 \leq \frac12|\nabla_ef|^2.
$$
We further use \eqref{eq:s} to have
\begin{align*}
\sum_{v_e^j\in I_2}|[v_e^i,v_e^j]f|^2\leq |\nabla_e f|^2|e_{kk}-e_{nn}|^2=2|\nabla_e f|^2. 
\end{align*}
We also use \eqref{eq:s4}-\eqref{eq:s5} to have
\begin{align*}
	\sum_{v_e^j\in I_3} |[v_e^i,v_e^j]f|^2
	&\leq \Big(\frac{(k-1)^2}{k-1+(k-1)^2}\1_{k\geq2}+\frac{n^2}{n-1+(n-1)^2}+\sum_{m=k}^{n-2}\frac1{m+m^2}\Big) |\nabla_ef|^2 
	\\
	&=\Big(2-\frac1k+\frac{1}{n-1}+\sum_{m=k}^{n-2}(\frac1m-\frac1{m+1})\Big)|\nabla_ef|^2=  2|\nabla_ef|^2.
\end{align*}
As a consequence, the result holds for $v_e^i=E_{kn}Q_e$ or $F_{kn}Q_e$. 

For $v_e^i=D_mQ_e$ as $[D_l,D_m]=0$ we also use \eqref{eq:s4}-\eqref{eq:s5}  to view $\{[v_e^i,v_e^j], j=1,\dots,\dg\}$ as a subset of orthonormal basis up to a factor with absolute value smaller than $\sqrt2$. We then have 
$$
\sum_j |[v_e^i,v_e^j]f|^2 \le2 |\nabla_ef|^2.
$$
Hence, the result follows. 
\end{proof}

We first prove the following lemma.
We write   $\bar e\sim e$ if $\bar e$ and $e$ appear in the same plaquette;
more precisely, if there exists $p\in \cP$ such that $\{e,e^{-1}\}\cap p \neq \emptyset$ and $\{\bar e,\bar e^{-1}\}\cap p \neq \emptyset$.

\bl\label{lem:zmm} 
Let  $\{v_e^i\}$ be the orthonormal basis given above.
For every $f\in  C^\infty(\cQ_L)$ and every $e\in E^+_{\Lambda_{L}}$,
one has
\begin{align*}
	|[v_e^i,\cL_L]f(Q)|  \leq \sum_{E_{\Lambda_L}^+\ni\bar e\sim e} a_{e,\bar e} |\nabla_{\bar e}f(Q)|,\qquad \forall Q\in \cQ_L,
\end{align*}
with $a_{e,\bar e}= N|\beta| \sqrt{\dg}$ for $e\neq \bar e$ and
$$
a_{e, e}= 2(d-1)N|\beta| (\sqrt{\dg}+\sqrt2N^{1/2}\gamma)\;,
$$
where $\gamma=1$ when $G=SO(N)$ and $\gamma=3\sqrt2$ for $G=SU(N)$.
\el
\begin{proof}
In this proof all the sums over $\bar e$ are restricted to $E_{\Lambda_L}^+$.
	Since the metric on $G$ is bi-invariant and each $v_e^i$ is right-invariant which generates a one-parameter family of isometries,
	$v_e^i$ commutes with the Beltrami-Laplacian $\Delta_e$. So we have
	\begin{align*}
		[v_e^i,\cL_L]f
		&=v_e^i \cL_L f-\cL_L v_e^i f
		\\
		&=\sum_{\bar e\sim e}  \big\<\nabla_{v_e^i} \nabla_{\bar e}\cS,\nabla_{\bar e} f \big\>
		+\big \<\nabla_e\cS \;, \; \nabla_{v_e^i}\nabla_{e}f-\nabla_{e}v_e^if \big\> \;.
	\end{align*}
Writing $\nabla_{\bar e}\cS = \sum_j(v_{\bar e}^j \cS) v_{\bar e}^j$, and  using \eqref{e:F27}, the first term on the RHS is equal to	
$$
\sum_{\bar e\sim e} \Big\<\sum_j (v_e^iv_{\bar e}^j\cS)  v_{\bar e}^j,\nabla_{\bar e} f \Big\>
		+\frac12\sum_j \Big\< (v_e^j\cS) [v_e^i,v_e^j],\nabla_{ e}f \Big\>\;.
$$		
For the second term we use $\nabla_{ e}f = \sum_j(v_{ e}^j f) v_{ e}^j$ and \eqref{e:F27} to write it as
\begin{align*}
&\sum_j\Big\<\nabla_e\cS,(v_e^iv_e^jf)v_e^j+v_e^jf\nabla_{v_e^i}v_e^j-(v_e^jv_e^if)v_e^j\Big\>
\\=&\sum_j (v_e^j\cS) \Big\<[v_e^i,v_e^j],\nabla_{ e}f \Big\>+\frac12 \sum_j v_e^j f \Big\<\nabla_e\cS,[v_e^i,v_e^j] \Big\>\;.
\end{align*}
Therefore,
	\begin{align*}
		[v_e^i,\cL_L]f
		&=\sum_{\bar e\sim e}
		\Big\<\sum_j (v_e^iv_{\bar e}^j\cS)  v_{\bar e}^j,\nabla_{\bar e} f \Big\>
		+\frac32\sum_j (v_e^j\cS) \Big\<[v_e^i,v_e^j],\nabla_{ e}f \Big\>
		\\
		& \qquad +\frac12 \sum_j v_e^j f \Big\<\nabla_e\cS,[v_e^i,v_e^j] \Big\>
		\; \eqdef  \; \sum_{k=1}^3I_k.
	\end{align*}
	
	For $I_1$, by similar calculation as in the proof of Lemma \ref{lem:4.1}, we have $|v_e^iv_{\bar e}^j\cS|\leq N|\beta|$ for $e\neq \bar e$;  also, $|v_e^iv_e^j\cS|\leq 2(d-1)N|\beta|$ since for each edge $e$ there are $2(d-1)$ plaquettes containing $e$ or $e^{-1}$. 
	Combining with H\"older's inequality we have
	\begin{align*}
		|I_1|
		&=
		\Big| \sum_{\bar e\sim e} \sum_j (v_e^iv_{\bar e}^j\cS)  v_{\bar e}^jf \Big|
		\leq \sum_{\bar e\sim e}\Big(\sum_j|v_e^iv_{\bar e}^j\cS|^2\Big)^{1/2}\Big(\sum_j|v_{\bar e}^jf|^2\Big)^{1/2}
		\\
		&\leq N|\beta| \sqrt{\dg} \sum_{\bar e\sim e,\bar e\neq e}|\nabla_{\bar e}f|+2(d-1)N|\beta| \sqrt{\dg} |\nabla_{ e}f|.
	\end{align*}

	For $I_2$ and $I_3$, fixing the edge $e$ we
recall our choice of the orthonormal basis $\{v_e^i\}_{1\le i\le \dg}$ above.
Using Lemma \ref{lem:com} we then have
	\begin{align*}
		|I_2+I_3|&\leq \frac32\Big(\sum_j |v_e^j\cS|^2\Big)^{1/2}\Big(\sum_j |[v_e^i,v_e^j]f|^2\Big)^{1/2}+\frac12\Big(\sum_j |v_e^jf|^2\Big)^{1/2}\Big(\sum_j |[v_e^i,v_e^j]\cS|^2\Big)^{1/2}\\
		 & \leq \sqrt 2 \gamma_1\Big(\sum_j |v_e^j\cS|^2\Big)^{1/2}\Big(\sum_j |v_e^jf|^2\Big)^{1/2}=\sqrt2\gamma_1|\nabla_e\cS||\nabla_ef|
		\\ & \leq 2\sqrt2(d-1)N^{3/2}\gamma|\beta||\nabla_ef|,
	\end{align*}
	where $\gamma_1=1$ for $G=SO(N)$ and $\gamma_1=3$ for $G=SU(N)$ and we use \eqref{e:DS} to bound $|\nabla_e\cS|$ by $2(d-1)N^{3/2}|\beta|\gamma/\gamma_1$ in the last inequality.
	Hence, the result follows.
\end{proof}


The next corollary together with uniqueness in Section~\ref{sec:uniqueYM}
proves Corollary~\ref{co:zmm}. 

\bc Suppose that Assumption \ref{ass1} holds. For $f, g\in C^\infty_{cyl}(\cQ)$, suppose that $\Lambda_f\cap \Lambda_g=\emptyset$.
Then one has
\begin{align*}
	\cov(f,g)\leq c_{1} \dg e^{-c_N d(\Lambda_f,\Lambda_g)}(\$ f \$_{\infty}\$ g \$_{\infty}+\|f\|_{L^2}\|g\|_{L^2}),
\end{align*}
where $c_{1}$ depends on $|\Lambda_f|$, $|\Lambda_g|$,
 and $c_N$ depends on $K_\cS$,  $N$ and $d$.
Here the covariance and $L^2$ are with respect to every tight limit of $\{\mu_{\Lambda_{L}, N, \beta}\}_L$. 
\ec
\begin{proof}
With the calculations and bounds obtained in the previous lemmas, 
together with our Poincar\'e inequality,
to prove exponential decay we can then apply an argument essentially from \cite[Section 8.3]{GZ}.
	We write $\mu_L=\mu_{\Lambda_{L}, N, \beta}$ for simplicity and consider
	\begin{align}\no
		\cov_{\mu_L}(f,g)
		&=\mu_L(fg)-\mu_L(f)\mu_L(g)=\mu_L(P_t^L(fg)) -\mu_L(P_t^Lf)\mu_L(P_t^Lg)
		\\
		&=\mu_L(P_t^L(fg)-P_t^Lf P_t^Lg)+\cov(P_t^Lf,P_t^Lg)\no
		\\
		&\leq\mu_L(P_t^L(fg)-P_t^Lf P_t^Lg)+\var(P_t^Lf)^{1/2}\var(P_t^Lg)^{1/2}.\label{zmm2}
	\end{align}
	Recall that the Poincar\'e inequality is equivalent to the following:
	$\var(P_t^Lf)\leq e^{-2tK_\cS}\|f\|_{L^2(\mu_L)}^2$ (see Remark \ref{re1}).
	Therefore by the Poincar\'e inequality,
	the last term in \eqref{zmm2} is bounded by
	\begin{align}\label{zmm3}
		\var(P_t^Lf)^{1/2}\var(P_t^Lg)^{1/2}\leq	e^{-2tK_\cS}\|f\|_{L^2(\mu_L)}\|g\|_{L^2(\mu_L)}.
	\end{align}
	 
	As $\cL_L$ is uniform elliptic operator with smooth coefficient, by H\"ormander's Theorem (c.f. \cite[Theorem~2.3.3]{MR2200233}) $P_t^Lf\in C^\infty(\cQ_L)$.
	Now we consider $P_t^L(fg)-P_t^Lf P_t^Lg$ in \eqref{zmm2} and we omit $L$ for notation simplicity.
	Recall that $P_t $ and $\cL $ commute on the domain $ D(\cL)$ (see e.g. \cite[Chap. I Exercise 1.9]{MaRockner}).
We have
	\begin{align*}
		P_t(fg)-P_tf P_tg
		&=\int_0^t \frac{\dif}{\dif s}[P_s(P_{t-s }f P_{t-s}g)]\dif s
		\\
		&=\int_0^t \Big [P_s\cL(P_{t-s }f P_{t-s}g)-P_s(\cL P_{t-s}f P_{t-s}g+ P_{t-s}f \cL P_{t-s}g) \Big]\dif s
		\\
		&=2\sum_{e}\int_0^tP_s\<\nabla_eP_{t-s }f, \nabla_eP_{t-s}g\>\dif s
		=2\sum_{e,i}\int_0^tP_s[(v_e^iP_{t-s }f)\cdot (v_e^iP_{t-s}g)]\dif s.
	\end{align*}
Here, to obtain the third line from the second line,  recalling the definition of $\cL$,
by $\nabla(fg)=g\nabla f +f\nabla g$ the first order terms cancel, and it then follows from  $\Delta(fg)=g\Delta f+f\Delta g+2\<\nabla f,\nabla g\>$. 
	Note that for every $e, i $, we have $(P_sv_e^i f)\cdot (P_sv_e^i g)=0$ since $\Lambda_f\cap \Lambda_g=\emptyset$.
	From this we then have
	\begin{align*}
		\sum_{e,i}(v_e^iP_{t-s }f)( v_e^iP_{t-s}g)
		&= \sum_{e,i}(v_e^iP_{t-s }f-P_{t-s}v_e^i f )\cdot( v_e^iP_{t-s}g-P_{t-s}v_e^i g)
		\\
		&+\sum_{e,i}(v_e^iP_{t-s }f-P_{t-s}v_e^i f)\cdot (P_{t-s}v_e^i g)
		\\
		&+\sum_{e,i}(v_e^iP_{t-s }g-P_{t-s}v_e^i g)\cdot(P_{t-s}v_e^i f)
		 \eqdef \sum_{e}(I_e^1+I^2_e+I^3_e).
	\end{align*}
	
	Suppose for the moment that we can prove the following:  for any $c>0$ and $f\in C^\infty_{cyl}(\cQ_L)$, there exists $B>0$ such that for $d(e,\Lambda_f)\geq Bt$  one has
	\begin{align}\label{eq:com}
		\sum_i\|v_e^iP_{t }f-P_{t}v_e^i f\|_{L^\infty}\leq \dg e^{-2cd(e,\Lambda_f)}\$ f \$_\infty.
	\end{align}
	We choose $t\sim d(\Lambda_f,\Lambda_g)/B$ below. Applying \eqref{eq:com}  to the function $g$
	with $e\in \Lambda_f$ (in which case $I_e^2=0$ since $v_e^i g=0$) and using \eqref{eq:nap}
	\begin{align*}
		\|I_e^1+I_e^3\|_{L^\infty}
		&\leq \sum_i\|v_e^iP_{t-s}f\|_{L^\infty}\|v_e^iP_{t-s}g-P_{t-s}v_e^ig\|_{L^\infty}
		\\
		&\leq \dg e^{-2cd(\Lambda_f,\Lambda_g)} \$ f \$_{\infty}\$ g \$_\infty.
	\end{align*}
	Similarly for $e\in \Lambda_g$, $I_e^3=0$ and
	\begin{align*}
		\|I_e^1+I_e^2\|_{L^\infty}\leq  \dg e^{-2cd(\Lambda_g,\Lambda_f)} \$ g \$_{\infty}\$ f \$_{\infty}.
	\end{align*}
	For $e\notin \Lambda_f\cup \Lambda_g$ we have $I_e^2=I_e^3=0$ and
	$d(e,\Lambda_f)\geq d(\Lambda_g,\Lambda_f)/2$ or $d(e,\Lambda_g)\geq d(\Lambda_g,\Lambda_f)/2$. For both cases we have
	\begin{align*}
		\|I_e^1\|_{L^\infty}
		\leq \dg e^{-cd(\Lambda_f,\Lambda_g)-c(d(e,\Lambda_f)\wedge d(e,\Lambda_g))} \$ f \$_{\infty}\$ g \$_{\infty}.
	\end{align*}
	
	With these bounds on $I_e^{1},I_e^{2},I_e^{3}$, we sum over $e$ and obtain that for $d(\Lambda_f,\Lambda_g)\geq Bt$
	\begin{align}\label{zmm4}
		\|	P_t(fg)-P_tf P_tg\|_{L^\infty}\leq c_1 \dg e^{-cd(\Lambda_f,\Lambda_g)}\$ f \$_{\infty}\$ g \$_{\infty}.
	\end{align}
	Substituting \eqref{zmm3} and \eqref{zmm4} into \eqref{zmm2} we get
	\begin{align*}
		\cov_{\mu_L}(f,g)
		\leq c_1 \dg e^{-cd(\Lambda_f,\Lambda_g)}\$ f \$_{\infty}\$ g \$_{\infty}+	e^{-2tK_\cS}\|f\|_{L^2}\|g\|_{L^2},
	\end{align*}
where $c_1$ depends on $|\Lambda_f|$ and $|\Lambda_g|$ and is independent of $L$.
	Since  $t\sim d(\Lambda_f,\Lambda_g)/B$, letting $L\to\infty$ the result follows.
	
	It remains to check the claimed bound \eqref{eq:com}. We use a similar argument as in \cite[Theorem 8.2]{GZ} which we adapt into our setting.  We have
	\begin{align}\label{zmm1}
		v_e^iP_tf-P_tv_e^if=\int_0^t\frac{\dif}{\dif s} \left( P_{t-s}v_e^iP_sf\right)\dif s=\int_0^tP_{t-s}[v_e^i,\cL]P_sf \dif s.
	\end{align}
	By Lemma~\ref{lem:zmm}, we have
	\begin{align*}
		\|[v_e^i,\cL]P_sf\|_{L^\infty}\leq \sum_{\bar e\sim e} a_{e,\bar e} \|\nabla_{\bar e} P_sf\|_{L^\infty},
	\end{align*}
	 for  constants $a_{e,\bar e}$ which are uniformly bounded in $e,\bar e$.  Hence, by \eqref{zmm1}
	\begin{align*}
		\sum_i\|v_e^i P_tf\|_{L^\infty}\leq \sum_i\|v_e^i f\|_{L^\infty}+\int_0^t\sum_{\bar e} D_{e,\bar e} \|\nabla_{\bar e} P_sf\|_{L^\infty}\dif s,
	\end{align*}
	with a matrix $D$ such that $D_{e,\bar e}=\dg a_{e,\bar e}$ if $e\sim \bar e$ and $D_{e,\bar e}=0$ for other case.
	Since $e\notin \Lambda_f$ we get $v_e^i f=0$ and by iteration
	\begin{align*}
		\sum_i\|v_e^i P_tf\|_{L^\infty}\leq\sum_{n=N_e}^\infty \frac{t^n}{n!}\sum_{\bar e}D_{e,\bar e}^{(n)} \sum_i\|v_{\bar e}^i f\|_{L^\infty},
	\end{align*}
	with $N_e=d(e,\Lambda_f)$ and $D_{e,\bar e}^{(n)}\leq C^n_0$ with $C_0=\dg (a_{e,e}+6(d-1)a_{e,\bar e})$. As a result, using $n!\geq e^{n\log n-2n}$, for $2-\log B+\log C_0+\frac{C_0}{B}\leq -2c$ and $d(e,\Lambda_f)\geq Bt$ we have
	\begin{align*}
	\sum_i\|v_e^i P_tf\|_{L^\infty}
	\leq\sum_{n=N_e}^\infty \frac{t^n}{n!}C_0^n \dg \$ f \$_{\infty}
	\leq \frac{(C_0t)^{N_e}}{N_e!}e^{tC_0} \dg \$ f \$_{\infty}
	\leq \dg e^{-2c d(e,\Lambda_f)}\$ f \$_{\infty}.
	\end{align*}
	Hence, \eqref{eq:com} follows.
\end{proof}

\br
From the above proof one can see 
$$
c_N\sim \frac{K_\cS}{\dg (a_{e,e}+6(d-1)a_{e,\bar e})},
$$ but this is not necessarily optimal.
\er

\section{Uniqueness of invariant measure}
\label{sec:uniqueYM}

In this section we prove Theorem \ref{th:1.3}.
As the results \eqref{eq:er} and \eqref{eq:er1}   in Theorem \ref{th:4.2} depend on $\rho_L$,
we cannot simply send $L\to\infty$ to conclude the result for $(P_t)_{t\geq0}$ on $\cQ$.  The idea of our proof is to construct a suitable coupling and find a suitable distance $\rho_{\infty,a}$ such that
for any $\mu,\nu\in \sP(\cQ_L)$, the Wasserstein distance w.r.t. $\rho_{\infty, a}$  between  $\mu P_t^L$ and $\nu P_t^L$ decays exponentially fast in time. Recall that $\rho_{\infty, a}$ is given in \eqref{e:rho-inf}  and we will choose a suitable parameter $a>1$ below.

We denote $C_{\Ric,N}= \frac{\alpha(N+2)}4-1$ which is a constant arising from  Ricci curvature in \eqref{e:Ricvv},
where  $\alpha=1, 2$ for $SO(N)$ and $SU(N)$ 
respectively.
For any $\mu, \nu\in \sP(\cQ)$, we introduce the Wasserstein distance
\begin{align*}
	W_p^{\rho_{\infty,a}}(\mu,\nu)\eqdef \inf_{\pi\in \sC(\mu,\nu)}\pi(\rho_{\infty,a}^p)^{1/p}.
\end{align*}

Recall that the generator $\cL_L$ is given by
	\begin{align}
	\cL_L F=	\sum_{e\in E_{\Lambda_{L}}^+}\Delta_{e}F+\sum_{e\in E_{\Lambda_{L}}^+}\<\nabla \cS(Q)_e,\nabla_{e} F\>.
\end{align}
For fixed $Q\in \cQ_L$ define
\begin{equ}[e:CD]
	C\eqdef\{(Q,Q'):Q'\in \text{cut}(Q)\}, \quad D\eqdef \{(Q,Q):Q\in \cQ_L\},
\end{equ}
where $\text{cut}(Q)$ consists of conjugate points of $Q$ and points having more than one minimal geodesics to $Q$.

In the following we prove the result for any $a>1$.

\bl\label{lem:4.6}
Suppose that $\widetilde K_\cS\eqdef C_{\Ric,N}-{(4+4\sqrt{a})}N|\beta|(d-1)>0$.
Then for every $L\in \mZ$,
\begin{align*}
	W_2^{\rho_{\infty,a}}(\mu P_t^L,\nu P_t^L)\leq e^{-\widetilde K_\cS t}W_2^{\rho_{\infty,a}}(\mu,\nu),\qquad t\geq0, \quad\mu, \nu\in \sP(\cQ_L).
\end{align*}
Here we  use periodic extension to view every measure as a probability on $\cQ$.
\el
\begin{proof}
To prove the statement we will construct a suitable coupling $(Q(t),Q'(t))_{t\geq0}$ between the two Markov processes associated to the generator $\cL_L$ starting from two different points $(Q,Q')$.
We will then use It\^o's fomula to calculate $\dif \rho_{\infty,a}^2(Q(t),Q'(t))$ and obtain
	\begin{align}\label{eq:rhoi}
		\rho_{\infty,a}^2(Q(t),Q'(t))\leq e^{-2\tilde{K}_St}\rho_{\infty,a}^2(Q(0),Q'(0)),\quad t\geq0.
	\end{align}
	Suppose that \eqref{eq:rhoi} holds and we use $\mP_t^{Q,Q'}$ to denote
	the distribution of the coupling $(Q(t),Q'(t))$. Then for any $\mu,\nu\in \sP(\cQ_L)$ and $\pi\in \sC(\mu,\nu)$ we set
	\begin{align*}
		\pi_t\eqdef \int \mP_t^{Q,Q'}\pi(\dif Q,\dif Q')\in \sC (\mu P_t^L,\nu P_t^L).
	\end{align*}
	Hence, for $t\geq0$
	\begin{align*}
		W_2^{\rho_{\infty,a}}(\mu P_t,\nu P_t)^2\leq \int\rho_{\infty,a}^2\dif \pi_t \leq e^{-2\widetilde{K}_{\cS} t}\pi(\rho_{\infty,a}^2),
	\end{align*}
	and the result follows.
In the following we prove \eqref{eq:rhoi} in three steps.

	\newcounter{MM} 
\refstepcounter{MM} 

\medskip

{\sc Step} \arabic{MM}.\label{MM1}\refstepcounter{MM} Construction of coupling $(Q(t),Q'(t))_{t\geq 0}$ and calculation of $\dif \rho^2(Q_e(t),Q'_e(t))$.

 The usual coupling for Brownian motions and diffusions  on  Riemannian manifolds
  is the Kendall--Cranston's coupling (c.f. \cite{Ken86}). In our case
  we adapt a construction in  \cite[Proposition~2.5.1]{Wang} to cancel the noise part,
 with one of the key modifications  due to our new weighted distance on our product manifold.

 More precisely,
	let $(Q(t),Q'(t))$ be the coupling on $\cQ_L\times \cQ_L$ starting from $(Q,Q')$ given by  the  following generator
	\begin{equ}[e:Lc]
		\cL^c
		=\sum_{e\in E_{\Lambda_{L}}^+}\Delta_{Q_e}
		+\sum_{e\in E_{\Lambda_{L}}^+}\Delta_{Q_e'}
		+2\sum_{i,j=1}^{\dim \q_L}\Big\<P_{Q,Q'}v_i,v'_j \Big\>_{T_{Q'}\cQ_L}v_iv'_j
		+\nabla\cS(Q)+ \nabla\cS(Q'),
	\end{equ}
where $\Delta_{Q_e}f(Q,Q')=(\Delta_e f(\cdot,Q'))(Q)$, $\Delta_{Q_e'}f(Q,Q')=(\Delta_e f(Q,\cdot))(Q')$ and  
$\{v_i\}, \{v_j'\}$ are orthonormal bases of tangent spaces at $Q$ and $Q'$,
 and $P_{Q,Q'}:T_Q\cQ_L\to T_{Q'}\cQ_L$ is the parallel translation along the geodesic from $Q$ to $Q'$.
 It is easy to see that $\cL^c$ is independent of the choices of the basis $\{v_i\}, \{v_j'\}$. In fact, to construct such coupling we need to avoid the cut locus $C$ and the diagonal set $D$ by suitable cut-off approximation and we refer to Appendix~\ref{sec:coupling} and \cite[Section 2.1]{Wang} for more details on the construction.

We intend to apply It\^o's formula to $\rho_e^2$ with $\rho_e\eqdef \rho(Q_e,Q_e')$.
To this end, we consider the projection map $\pi_e: \cQ_L \to G$ defined by $\pi_e Q\eqdef Q_e$.
We then
write $\hat\rho_e( \;\cdot \;  ; Q_e')$ for the pull-back of the function
$\rho(\cdot,Q_e')$ via the map $\pi_e$.
Namely, fixing any $ Q_e' \in G$, the function
$\hat\rho_e( \;\cdot \;  ; Q_e')$ is a function on $\cQ_L$ defined by
$$
	\hat\rho_e(Q ; Q_e')
	\eqdef \rho(\pi_e Q , Q_e') = \rho(Q_e , Q_e')
	\qquad
	\mbox{for } Q\in \cQ_L.
$$
Similarly we  define  function $\hat\rho_e(  Q_e ; \;\cdot \;  )$ on $\cQ_L$ as
$$
	\hat\rho_e(Q_e ; Q')
	\eqdef \rho(Q_e , \pi_e Q') = \rho(Q_e , Q_e')
	\qquad
	\mbox{for } Q'\in \cQ_L.
$$	
We can also write $\rho_e=\rho(\pi_eQ,\pi_eQ')$ and view $\rho_e$ as a function on $\cQ_L\times \cQ_L$.
	
For $R\in \mN$, we choose a smooth cut-off function $\chi_R:[0,\infty)\to [0,\infty)$ satisfying $\chi_R(x)=x$ for $x\geq 1/R$ and $\chi_R|_{[0,\frac1{2R}]}=0$ and $\chi_R'\geq0$. 

Since $\rho_e^2$ is smooth near the diagonal,	we claim that by It\^o's formula (see \cite[Section 2.1]{Wang}, \cite[Section 6.5]{Hsu}),
and writing $\rho_e(t)=\rho(Q_e(t),Q'_e(t))$,
 we have
	\begin{align}\label{eq:Ito}
		\dif \chi_R(\rho^2_e(t))\leq 2\chi_R'(\rho_e^2(t)) \rho_e(t)J(Q(t),Q'(t))\dif t
	\end{align}
for  $t<T \eqdef \inf\{t\geq 0,Q(t)=Q'(t)\}$ 
where
$J $ is a continuous function on $\cQ_L\times \cQ_L$
such that $J\geq I_{\cS}$ on $(D\cup C)^c$.
 Here
\begin{align}\label{e:IS}
	I_{\cS}(Q,Q') \eqdef I(Q_e,Q_e')
	+\Big((\nabla \cS)\hat\rho_e(\; \cdot \; ; Q_e')\Big)(Q)
	+\Big((\nabla \cS)\hat\rho_e(Q_e; \; \cdot\;)\Big)(Q')\;,
\end{align}
and $I(x,y)$ is the index along  $\gamma:[0,\rho(x,y)]\to G$ which is the minimal geodesic from $x$ to $y$ in $G$:
\begin{align*}
	I(x,y)\eqdef\sum_{i=1}^{\rm{dim}\mfg-1}\int_0^{\rho(x,y)}\Big(|\nabla_{\dot{\gamma}}J_i|^2-\<\sR(J_i,\dot \gamma)\dot\gamma,J_i\>\Big)_s\dif s,
\end{align*}
where $\{J_i\}_{i=1}^{\rm{dim}\mfg-1}$ are Jacobi fields along $\gamma$ such that at $x$ and $y$, they, together with $\dot\gamma$, form an orthonormal basis.
Note that
the reason to derive a bound in terms of  $J$ instead of $I_\cS$ in \eqref{eq:Ito}
is that
$J$   
is defined everywhere on  $\cQ_L\times \cQ_L$  whereas $I_\cS$ is not well-defined on $C\cup D$.
 In Step \ref{MM2} below we control $I_\cS$ by a continuous function on $\cQ_L\times \cQ_L$, which can also control $J$.

The rigorous derivation of \eqref{eq:Ito} follows by cut-off approximation  to avoid the cut locus $C$ and the diagonal set $D$ (c.f. \cite[Theorem 2.1.1]{Wang}, \cite[Theorem 6.6.2]{Hsu}).
In the following we give the idea on how the terms in \eqref{e:IS} arise and we put more details of the construction and derivation of \eqref{eq:Ito} in Appendix~\ref{sec:coupling}.

 For $t<T$ and $(Q(t), Q'(t))\notin C\cup D$,  on the support of $\chi_R(\rho_e^2)$,
 $I_\cS(Q,Q')$ is given by    $\cL^c \rho_e$. 
To prove this,
since $\rho_e= \rho(\pi_eQ,\pi_eQ')=\rho(Q_e,Q_e')$ only depends on $Q_e, Q_e'$
(i.e. independent of the values of $Q,Q'$ on the other edges),
we can write the first three terms in  $\cL^c \rho_e$  
as (see the RHS of \eqref{e:Lc})
 \begin{align}\label{eq:d1}
 \Delta_{Q_e}\rho(Q_e,Q_e')
 +\Delta_{Q_e'}\rho(Q_e,Q_e')
 +2\sum_{i,j=1}^{\dg}
    \Big\< P_{Q_e,Q'_e}v_{e,i} \; ,\; v'_{e,j}\Big\>_{T_{Q'_e}G}
     v_{e,i}v'_{e,j} \, \rho(Q_e,Q_e'),
 \end{align}
with $\{v_{e,i}\}$ and $\{v_{e,j}'\}$ being an orthonormal basis
of the tangent space at $Q_e, Q_e'$ and  $P_{Q_e,Q'_e}:T_{Q_e}G\to T_{Q'_e}G$ being the parallel translation along the geodesic from $Q_e$ to $Q'_e$.  Here we used the fact that
$$
P_{Q_e,Q'_e}  (v_e)   
=
 (P_{Q,Q'}  v)_e
\qquad
\forall v\in T_{Q} \cQ_L
$$
(in particular the $e$ component of the geodesic from $Q$ to $Q'$ is the geodesic from $Q_e$ to $Q_e'$.)
 By the second variational formula  (c.f. \cite[p21-22]{CE75}, \cite[Theorem~2]{Ken86}, \cite[Lemma 6.6.1]{Hsu}) we know that  \eqref{eq:d1} is equal to $I(Q_e,Q_e')$.

Moreover,
the last two terms involving $\nabla \cS$ in $\cL^c$ give rise to the last two terms in \eqref{e:IS}.

The quadratic variation of the martingale part from applying It\^o's formula  to $\rho_e$ is
\begin{align}\label{eq:gamma}
	|\<\dot\gamma,\dot \gamma\>(Q_e')-\<\dot\gamma,\dot \gamma\>(Q_e)|^2+\sum_{i=1}^{\dg-1}|\<J_i,\dot \gamma\>(Q_e')-\<J_i,\dot \gamma\>(Q_e)|^2
\end{align}
by the first variation formula (c.f. \cite[p5]{CE75}, \cite[Section 6.6]{Hsu}).
Since $\{J_i\}_{i=1}^{\rm{dim}\mfg-1}$  together with $\dot\gamma$
form an orthonormal basis, each term in \eqref{eq:gamma} is zero,
which implies that the martingale part is zero. We also refer to the derivation of \eqref{eq:va} in Appendix~\ref{sec:coupling} for more details on the calculation of the quadratic variation.

	{\sc Step} \arabic{MM}.\label{MM2}\refstepcounter{MM} Estimate the RHS of  \eqref{e:IS}.

In this step we estimate the RHS of  \eqref{e:IS} and prove that for $t<T$
\begin{align}\label{b:rhoe}
	\p_t \rho_e^2
	\leq
	-2C_{\Ric,N} \rho_e^2
	+2N|\beta| \sum_{p,p\succ e}\rho_e
	\Big(\rho_e+\sum_{e\neq \bar e\in p}\rho_{\bar e}\Big).
	\end{align}

	By the index lemma  (see \cite[Theorem 2.1.4]{Wang} or \cite[Lemma 6.7.1]{Hsu}),
 for $x=Q_e$, $y=Q_e'$ with $\gamma:[0,\rho_e]\to G$ the minimal geodesic  from $Q_e$ to $Q_e'$, where we recall
 that $\rho_e = \rho(Q_e,Q_e')$,  we have
	\begin{align}\label{eq:zb1}
		I(Q_e,Q_e')\leq -\int_0^{\rho_e}\Ric(\dot \gamma,\dot \gamma)\dif s=-C_{\Ric,N}\rho_e.
	\end{align}
	
In the following we consider the last two terms in \eqref{e:IS}.
Given $Q,Q'\in \cQ_L$ as above, we define a path $\Gamma: [0,\rho_e]\to \cQ_L$
which goes from $Q$ to $Q'$
as follows.
For any ${\bar e}\in E^+_{\Lambda_{L}}$, we can find a geodesic $\gamma^{\bar e} : [0,\rho_{\bar e}] \to G_{\bar e}$
from $Q_{\bar e} $ to $Q_{\bar e}'$. Here $\rho_{\bar e}$ is the length of the geodesic.
 We then set
\begin{align*}
	\Gamma(s)  =  \Big(\tilde\gamma^{\bar e}(s)\Big)_{\bar e\in E^+_{\Lambda_{L}}} \in \cQ_L
	\qquad  (s\in [0,\rho_e])
	\qquad
	\mbox{where}
	\qquad
	\tilde\gamma^{\bar e} (s) = \gamma^{\bar e} (\rho_{\bar e} s / \rho_e).
\end{align*}
We can check that we indeed have
$\Gamma(0) = Q$ and $\Gamma(\rho) = Q'$ and $\Gamma$ is the geodesic from $Q$ to $Q'$. Also, we have
$$
\pi_e (\Gamma(s)) = \gamma(s)\qquad  (\forall s\in [0,\rho_e]).
$$
With the above notation at hand, we write the last two terms in \eqref{e:IS} as
\begin{align}\label{dsds}
\Big((\nabla \cS)\hat\rho_e(\; \cdot \; ; Q_e')\Big)(Q)
	+\Big((\nabla \cS)\hat\rho_e(Q_e; \; \cdot\;)\Big)(Q')
	=
	\<\nabla \cS,\dot \gamma\>(Q')-\<\nabla \cS,\dot\gamma\>(Q),
\end{align}
with $\dot \gamma$ extended as a tangent vector  field on  $\cQ_L$ along the curve $\Gamma$,
which is still denoted by $\dot \gamma$, by setting all the other components as zero.

We then write  \eqref{dsds} as
\begin{align*}
	\int_0^{\rho_e}\Big(\frac{\dif}{\dif s} \Big\<\nabla \cS(\Gamma(s)),\dot \gamma(s)\Big\>\Big)\dif s
	&=\int_0^{\rho_e} \Big( \dot\Gamma\<\nabla \cS,\dot \gamma\> \Big) (\Gamma(s))\dif s.
\end{align*}

Hence, we get
$$
\eqref{dsds} = \int_0^{\rho_e}   \Big( \dot\Gamma\<\nabla \cS,\dot \gamma\> \Big) (\Gamma(s))\dif s.
$$
 Below we estimate the above integral. With a slight abuse of notation,
for an edge $e\in E^+_{\Lambda_L}$ we write $e\in p$ if $\{e,e^{-1}\}\cap p \neq \emptyset$, namely we view edges as undirected in the calculation below.
We also extend $\dot{\tilde\gamma}^{\bar e}, \bar e\in E^+_{\Lambda_L}$  as tangent vector field on $\cQ_L$ along $\Gamma$, which is still denoted by $\dot{\tilde\gamma}^{\bar e}$, by setting all the  components other than ${\bar e}$ to be zero.
Then recalling our formula for $\cS$ we have
\begin{align}
	\int_0^{\rho_e}  \Big( \dot\Gamma\<\nabla \cS,\dot \gamma\> \Big) (\Gamma(s))\dif s\no
	&=N\beta\sum_{p\succ e}\sum_{\bar e\in p}
	\int_0^{\rho_e}\dot{\tilde{\gamma}}^{\bar e}(\dot\gamma \Re\tr(Q_p)) \dif s\no
	\\
	&\leq N|\beta|\sum_{p\succ e}\sum_{\bar e\in p}\int_0^{\rho_e}|\dot{\tilde{\gamma}}^{\bar e}||\dot\gamma| \, \dif s\no
	\\
	&\leq N|\beta| \sum_{p\succ e}
	\Big(\rho_e+ \sum_{e\neq\bar e\in p}\rho_{\bar e}\Big),\label{eq:za1}
\end{align}
where we used $|\dot \gamma|=1$ and $|\dot{\tilde{\gamma}}^{\bar e}|=\rho_{\bar e}/\rho_e$.
Here we  calculate
$\dot{\tilde{\gamma}}^{\bar e}(\dot\gamma \Re\tr(Q_p))$
as follows: for $Q_p=Q_eQ_{\bar e}Q_1Q_2$ with $Q_1, Q_2\in G$ we get
\begin{align*}
\dot{\tilde{\gamma}}^{\bar e}
(\dot\gamma \Re\tr(Q_p))
&=\frac{\dif}{\dif t}\Big|_{t=0}\frac{\dif}{\dif s}\Big|_{s=0}
\Re\tr\Big(\gamma(s) {\tilde{\gamma}}^{\bar e}(t) Q_1Q_2\Big)
\\
&=\Re\tr\Big(\dot\gamma \dot{\tilde{\gamma}}^{\bar e}Q_1Q_2\Big),
\end{align*}
the absolute value of which by H\"older's inequality for trace is bounded by $\, |\dot{\tilde{\gamma}}^{\bar e}| \, |\dot\gamma|$. Similar calculation holds for $Q_p=Q_eQ_1Q_{\bar e}Q_2$ and $Q_p=Q_eQ_1Q_2Q_{\bar e}$ and we use similar argument as in the proof of Lemma \ref{lem:4.1} to control $\dot \gamma\dot \gamma \Re\tr(Q_p)$ by $|\dot \gamma|^2$.
Hence, by  \eqref{eq:Ito}, \eqref{e:IS}, \eqref{eq:zb1}, \eqref{eq:za1}, we get
\begin{align*}
	\p_t\chi_R(\rho_e^2)
	\leq -2C_{\Ric, N}\, \chi_R'(\rho_e^2)\,\rho_e^2
	+2\chi_R'(\rho_e^2) N|\beta| \sum_{p,p\succ e}\Big(\rho_e^2+\sum_{e\neq\bar e\in p}\rho_{\bar e}\rho_e\Big).
\end{align*}
Letting $R\to \infty$ and by dominated convergence theorem and the fact that $\chi_R'$ is uniformly bounded in $R$,
\eqref{b:rhoe} holds.


\vspace{1ex}

	{\sc Step} \arabic{MM}.\label{MM3}\refstepcounter{MM} Derivation of \eqref{eq:rhoi}.
	
We extend $(Q(t),Q'(t))_{t\geq0}$ periodically
as a process on $\cQ\times \cQ$, which is still denoted as $(Q(t),Q'(t))_{t\geq0}$. \eqref{b:rhoe} also holds for the extension.
By \eqref{b:rhoe} we have
\begin{equ}
	\p_t \rho_e^2
	\leq
	-2C_{\Ric,N}\, \rho_e^2
	+2N|\beta| \Big(2(d-1)\rho_e^2+\sum_{p,p\succ e}\sum_{e\neq\bar e\in p}\rho_{\bar e}\rho_e\Big).      
\end{equ}
In the following we bound $\rho_{\bar e}\rho_e$. To obtain the desired rate given by $\widetilde{K}_\cS t$, we  need to control $\rho_{\bar e}\rho_e$ in different ways  depending on the relations
 between $|e|$ and $|\bar e|$. We first fix a plaquette $p$ and consider two edges $\bar e\neq e$. 


\noindent For the edges satisfying $|e|= |\bar e|$ \footnote{If two edges $e\neq \bar e$ share the same vertex and this vertex is  closer to origin,  we may have $|e|= |\bar e|$.} we have
\begin{align*}
	2\rho_{\bar e}\rho_e\leq \rho_{\bar e}^2+\rho_e^2.
\end{align*}
For the edges satisfying $|e|\neq |\bar e|$  we have
\begin{align*}
	\frac2{\sqrt{a}}\rho_{\bar e}\rho_e\leq \frac1a\rho_{\bar e}^2+\rho_e^2.
\end{align*}
The reason for the choice of the above weight is as follows: there is one plaquette $p$ such that only one edge $\bar e\neq e$ in  $p$ with the same distance as $|e|$ and other edges  with the distance larger than $|e|$.
Thus, since for each edge $e$ there are
$2(d -1)$ plaquettes in $\cP$ such that $p\succ e$, we get
\begin{align*}
&2N|\beta|\sum_{p,p\succ e}\sum_{e\neq\bar e\in p}\rho_{\bar e}\rho_e
\\
&\leq \sqrt{a}N|\beta|  \sum_{p,p\succ e}  \sum_{|e|\neq|\bar e|\in p}\Big(\frac1a\rho_{\bar e}^2+\rho^2_e\Big)
+N|\beta|\sum_{p,p\succ e}\sum_{|e|=|\bar e|\in p,e\neq\bar e}
\Big(\rho_{\bar e}^2+\rho^2_e\Big)
\\
&=\sqrt{a}N|\beta|\sum_{p,p\succ e}\sum_{|e|\neq|\bar e|\in p}\frac1a\rho_{\bar e}^2+(4\sqrt{a}+2)(d-1)N|\beta|\rho^2_e+N|\beta|\sum_{p,p\succ e}\sum_{|e|=|\bar e|\in p,e\neq\bar e}\rho_{\bar e}^2
\end{align*}
where the first sum for $\bar e$ with $|e|\neq|\bar e|$ includes two edges and the second sum for $\bar e$ with $|e|=|\bar e|$ contains only one edge.
Note that we also get an extra $\frac1a$ before $\rho_{\bar e}^2$ with $|\bar e|=|e|+1$, which can be put into the weight $\frac1{a^{|\bar e|}}$. Substituting the above calculation into \eqref{b:rhoe} and using again the fact that for each edge $e$ there are
$2(d -1)$ plaquettes in $\cP$ such that $p\succ e$  we get
\begin{align*}
	\frac1{a^{|e|}}\p_t\rho_e^2\leq& -2C_{\Ric,N} \frac1{a^{|e|}}\rho_e^2+(4\sqrt{a}+6)N|\beta|(d-1)\frac1{a^{|e|}}\rho_e^2
	\\
	&+N|\beta|\sum_{p,p\succ e}
	\Big(\sqrt{a}\sum_{|e|\neq|\bar e|\in p}\frac1{a^{|\bar e|}}\rho_{\bar e}^2+\sum_{|e|=|\bar e|\in p,e\neq\bar e}\rho_{\bar e}^2\Big).
\end{align*}
Taking sum over $e$ we notice that $\rho_e^2$ also appears when calculating $\frac1{a^{|\bar e|}}\p_t\rho_{\bar e}^2$ with $\bar e$ and $e$ in the same plattque, which at most gives $2\sqrt aN|\beta| \frac1{a^{|e|}}\rho_e^2$ and $N|\beta| \frac1{a^{|e|}}\rho_e^2$ from $\frac1{a^{|\bar e|}}\p_t\rho_{\bar e}^2$ with $|\bar e|\neq|e|$ and  $|\bar e|=|e|$, respectively. Since for each edge $e$ there are
$2(d -1)$ plaquettes in $\cP$ such that $p\succ e$, we get
\begin{equ}[e:gronwa]
	\sum_{e\in E^+}\frac1{a^{|e|}}\p_t\rho_e^2\leq -2C_{\Ric,N} \sum_{e\in E^+}\frac1{a^{|e|}}\rho_e^2+(8+8\sqrt{a})N|\beta|(d-1)\sum_{e\in E^+}\frac1{a^{|e|}}\rho_e^2.
\end{equ}
Hence, \eqref{eq:rhoi} follows from Gronwall's lemma.
\end{proof}


Now we prove Theorem \ref{th:1.3}. 
One of the important ingredients 
in the proof 
is that under Assumption \ref{ass1},
 the condition of Lemma~\ref{lem:4.6}
can {\it indeed} be satisfied
by tuning the weight parameter $a>1$ to be sufficiently close to $1$,
see Eq.~\eqref{e:tildeK} below.
The crucial reason for this proof to work
is that the last term in our bound \eqref{e:gronwa}
is of order $N$,
rather than $N^p$ for some $p>1$.
This is a nontrivial point:
indeed, that term comes from bounding the $\nabla \cS$ terms
on the right-hand side of  \eqref{e:IS},
but note that 
 $\cS$ defined in \eqref{e:defS} would appear to 
be of order $N^2$ if one naively bound $\tr(Q_p) \le N$, in which case the proof would break down.
In fact in the previous proof we instead apply the property of the Lie group $G$ and H\"older inequality to separate different vector fields appearing in the second order derivative of $\cS$, which could finally be bounded by sum of Riemannian distances up to a factor $N|\beta|$.

\begin{proof}[Proof of Theorem \ref{th:1.3}]
For any two invariant measures $\mu, \nu$ of \eqref{eq:YM in}, 
we can find two sequences $\{\mu_L\}, \{\nu_L\}\subset \sP(\cQ_L)$
 such that their periodic extensions over the entire $\cQ$, which are still denoted by $\mu_L,\nu_L$, converge to $\mu, \nu$ weakly in $\cQ$, 
 with the distance induced by $\|\cdot\|$ defined in \eqref{eq:norm}.
 Indeed, let $Q(0):\Omega\to \cQ$ be a random variable
 such that $\mbox{Law}(Q(0))=\mu$, and then define
$\mu_L \eqdef \mbox{Law}(Q^L(0))$ where $Q^L(0):\Omega\to \cQ_L$ is given by
 \begin{equ}
 	Q_e^L(0)=
 	\begin{cases}
 		\displaystyle
 		Q_e(0)&\qquad
 		e\in E^+_{\Lambda_{L-1}} 
 		\\
 		\displaystyle I_N\;,
 		&\qquad
 		e\in E^+_{\Lambda_L}\backslash E^+_{\Lambda_{L-1}}\;,
 	\end{cases}
 \end{equ}
then $\{\mu_L\}$ satisfy the desired property. 
The sequence $\{\nu_L\}$ can be constructed in the same way.

By Lemma \ref{lem:exist} we obtain  the unique solutions $Q^L\in C([0,\infty);\cQ_L)$ to \eqref{eq:YM}  starting from the initial distribution $\mu_L\in \sP(\cQ_L)$. 
By periodic extension we view $Q^L\in C([0,\infty);\cQ)$.  Recall that $(P_t^L)_{t\geq0}$ is the Markov semigroup associated with the solution to \eqref{eq:YM}. By global well-posedness of \eqref{eq:YM}, 
we obtain for $F\in C^\infty_{cyl}(\cQ)$ and $t\geq0$
 $$\int P_t^LF\dif \mu_L=\E F(Q^L(t)).$$
Similarly, using Lemma \ref{lem:4.7} we obtain unique solutions $Q\in C([0,\infty);\cQ)$ to \eqref{eq:YM in} starting from the initial distribution $\mu$.
Recall that $(P_t)_{t\geq 0}$ is the Markov semigroup for the Markov process associated to \eqref{eq:YM in}. By uniqueness in law of the solution to \eqref{eq:YM in} we have
$$\int P_tF\dif \mu=\E F(Q(t)).$$

As $\mu_L$ converges to $\mu$ weakly in $\cQ$,  by the same argument as in the proof Proposition~\ref{lem:4.7}, the law of $\{Q^L\}$ is tight in $C([0,\infty);\cQ)$ and the tight limit satisfies the limit equation \eqref{eq:YM in} with the initial distribution $\mu$. By  uniqueness in law  of equation \eqref{eq:YM in}, which follows from pathwise uniqueness in Proposition~\ref{lem:4.7} and Yamada--Watanabe Theorem,  the law of $Q^L$ converges weakly to the law of $Q$ in $C([0,\infty);\cQ)$, as $L\to \infty$.
As a result, for $F\in C^\infty_{cyl}(\cQ)$ we have
\begin{align}\label{eq:5.13}
	 \int P_t^LF\dif \mu_L=\E F(Q^L(t))\to \E F(Q(t))=\int P_tF\dif \mu, \quad L\to\infty.
\end{align}
Similarly, we obtain
$$
\int P_t^LF\dif \nu_L\to \int P_tF\dif \nu, \quad L\to\infty.
$$
Moreover, by the condition $K_\cS >0$,
there exists $a>1$ such that
\begin{equ}[e:tildeK]
\widetilde{K}_{\cS}=C_{\Ric,N}-(4+4\sqrt{a})N|\beta|(d-1)>0.
\end{equ}
We then invoke  Lemma \ref{lem:4.6} to have
\begin{align*}
	&\Big| \int F\dif \mu-\int F\dif \nu\Big|
	=\Big| \int P_tF\dif \mu-\int P_tF\dif \nu\Big|
	=\lim_{L\to\infty}\Big|	\int P_t^LF\dif \mu_L-\int P_t^LF\dif \nu_L\Big|
	\\
	&=\lim_{L\to\infty}
	\inf_{\pi\in \sC(\mu_LP_t^L,\nu_LP_t^L)}
	\Big|	\int (F(x)-F(y))\dif \pi(x,y)\Big|
	\leq C_F
	\lim_{L\to\infty}
	W_2^{\rho_{\infty,a}}(\mu_L P_t^L,\nu_L P_t^L)
	\\
	&\leq C_F e^{-\widetilde{K}_{\cS}t}\lim_{L\to\infty}W_2^{\rho_{\infty,a}}(\mu_L ,\nu_L )\leq C(a) e^{-\widetilde{K}_St}
\end{align*}
where $C_F$ only depends on $F$, 
 and the 
constant $C(a)$ is independent of $L$
by boundedness of $\rho_{\infty,a}$,
i.e. $\rho_{\infty,a}(Q,Q')<\infty$ for any $Q,Q'\in \cQ$.
Letting $t\to\infty$ we have 
\begin{align*}
	\Big|	\int F\dif \mu-\int F\dif \nu\Big|=0.
\end{align*}
Hence, $\mu=\nu$. This gives the uniqueness of invariant measure,
as denoted by $\muYM$ in the theorem.

By Theorem \ref{th:in1}, every tight limit is the invariant measure of \eqref{eq:YM in}. Hence, it is also unique and the second result of the theorem follows.

To prove the last statement \eqref{e:WnuP},
taking now an {\it arbitrary} probability measure $\nu$ on $\cQ$  we also have $\{\nu_L\}$  constructed similarly as above. 
We denote by $Q^{\nu_L}$ and $Q^\nu$ the processes starting from $\nu_L$ and $\nu$, respectively.  
We also have, as in \eqref{eq:5.13},
\begin{align}\label{eq:5.13again}
	 \int P_t^LF\dif \nu_L
	 =\E F(Q^{\nu_L}(t))
	 \to \E F(Q^\nu(t))
	 =\int P_tF\dif \nu, \quad L\to\infty.
\end{align}

Recall $\{\mu_L\}$ and $\mu=\muYM$ as the unique invariant measure given above. 
By triangle inequality and Lemma \ref{lem:4.6} we have for $t\geq0$
\begin{align}
	W_2^{\rho_{\infty,a}}(\nu P_t,\mu)
	&\leq W_2^{\rho_{\infty,a}}(\nu_L P_t^L,\nu P_t)
	+W_2^{\rho_{\infty,a}}(\nu_L P_t^L,\mu_L P_t^L)
	+W_2^{\rho_{\infty,a}}(\mu,\mu_L )\no
	\\
	&\leq W_2^{\rho_{\infty,a}}(\nu_L P_t^L,\nu P_t)
	+ e^{-\widetilde{K}_{\cS}t}W_2^{\rho_{\infty,a}}(\mu_L ,\nu_L )
	+W_2^{\rho_{\infty,a}}(\mu,\mu_L )  \no
	\\
	&\leq \E\rho_{\infty,a}^2(Q^{\nu_L}(t), Q^\nu(t))
	+C(a) e^{-\widetilde{K}_{\cS}t}
	+W_2^{\rho_{\infty,a}}(\mu,\mu_L ).\label{eq:c}
\end{align}
As $\cQ$ is compact w.r.t. the distance $\rho_\infty^a$, 
$Q^{\nu_L}(t)$ is tight in $(\cQ,\rho_\infty^a)$. 
Using \eqref{eq:5.13again} we then have for $t\geq0$
 $$\E\rho_{\infty,a}^2(Q^{\nu_L}(t), Q^\nu(t))\to0,\quad  L\to\infty.$$
Letting $L\to\infty$ in \eqref{eq:c}, we have
\begin{align*}
	W_2^{\rho_{\infty,a}}(\nu P_t,\mu)
	\leq C(a) e^{-\widetilde{K}_{\cS}t},
\end{align*}
which is  \eqref{e:WnuP}.
It is clear from \eqref{e:tildeK} that 
 $\widetilde{K}_{\cS}$   only depends on the constant $a$,  $d$, $\beta$ and dimension of $G$.
\end{proof}

 \appendix
\renewcommand{\appendixname}{Appendix~\Alph{section}}
\renewcommand{\theequation}{A.\arabic{equation}}
\section{Construction of coupling}
\label{sec:coupling}

In this appendix, we follow \cite{Ken86}, \cite[Section~2.1]{Wang}  to construct the coupling 
$$
(Q(t),Q'(t))_{t\geq0}
$$ 
starting from $(Q, Q')$ by approximation,
 and prove \eqref{eq:Ito}.
The coupling argument presented here is similar with  \cite[Chapter~2]{Wang} but a main difference is that  the above reference applies It\^o's formula to a distance on a given manifold -- which would be $\rho_L$ (not $\rho_e$) in our case, but we will
apply It\^o's formula to the quantity $\chi_R(\rho_e^2)$.

Before proceeding we recall the basic definitions and notations (c.f. \cite[Chapter~2]{Hsu} for more detailed explanations).
Recall that  $\cQ_L$ is a Riemannian manifold with dimension $d=|E^+_{\Lambda_L}|\dg$.
Let $\sO(\cQ_L)$ be the orthonormal frame bundle over $\cQ_L$, which is  a $d(d+1)/2$-dimensional Riemannian manifold.
Given $l\in \mR^d$, let $H_l$ be the corresponding horizontal vector field on $\sO(\cQ_L)$.
Denote by $\pi: \sO(\cQ_L) \rightarrow \cQ_L$ be the canonical projection.
For any $\Phi\in \sO(\cQ_L)$ we have $\Phi l\in T_{\pi \Phi}\cQ_L$ and $H_l(\Phi)\in T_{\Phi}\sO(\cQ_L)$ is the horizontal lift of $\Phi l\in T_{\pi \Phi}\cQ_L$ to $\Phi$.
In particular, let $\{l_i\}_{i=1}^d$ be an orthonormal basis of $\mR^d$, define the horizontal Laplace operator
$$\Delta_{\sO(\cQ_L)}\eqdef \sum_{i=1}^dH^2_{l_i},$$
which is independent of the choice of the basis $\{l_i\}$.
Moreover, for any vector field $Z$ on $\cQ_L$ we define its horizontal lift by
$H_{\Phi}Z\eqdef H_{\Phi^{-1}Z}(\Phi)$ for $\Phi\in \sO(\cQ_L)$, where $\Phi^{-1}Z$ is the unique vector $l\in \mR^d$ such that $Z_{\pi \Phi}=\Phi l$.

As in \cite[Chapter~6]{Hsu}, for a function $f$ defined on $\sO(\cQ_L)\times \sO(\cQ_L)$,
we denote by $H_{l_i,1}f$ and $H_{l_i,2}f$  the derivatives of $f$ with respect to the horizontal vector field $H_{l_i}$ on the first and the second variable respectively. The horizontal Laplacian on the first and the second variable are
$$\Delta_{\sO(\cQ_L),1}=\sum_{i=1}^dH_{l_i,1}^2,\quad \Delta_{\sO(\cQ_L),2}=\sum_{i=1}^dH_{l_i,2}^2.$$


{\bf Construction of coupling.}
Consider the following Stratonovich SDE with $Q(t)\eqdef \pi(\Phi_t)$
\begin{equation}\label{eq2.1}
	\dif \Phi_t=\sum_{i=1}^d H_{\Phi_t}(\Phi_t)\circ \dif N_t,\qquad
	\dif N_t=\sqrt2\dif B_t+\Phi_t^{-1}\nabla \cS(Q(t))\dif t,
\end{equation}
where  $(B_t)_{t\ge 0}$ is a standard
$d$-dimensional Brownian motion and $\pi \Phi_0=Q$. Then $Q(t)$ is an $\cL_L$-diffusion process ($\cL_L$ as in \eqref{eq:L})
 starting from $Q$ and $\Phi_t$ is called its horizontal lift.

As the Riemannian distance is not smooth on $C$ and $D$  defined in \eqref{e:CD}, we introduce cut-off approximation as follows:
For any $n\geq 1$ and $\eps\in (0,1)$, let $h_{n,\eps}\in C^\infty(\cQ_L\times \cQ_L)$ such that $0\leq h_{n,\eps}\leq {1-\eps}, h_{n,\eps}|_{C_n^c}={1-\eps}$ and $h_{n,\eps}|_{C_{2n}}=0$, where
\begin{align*}
	C_n\eqdef \Big\{(Q,Q'):\rho_{\cQ_L\times \cQ_L}((Q,Q'),C)\leq \frac1n\Big\},\quad n\geq 1,
\end{align*}
with $\rho_{\cQ_L\times \cQ_L}$ the Riemannian distance on $\cQ_L\times \cQ_L$. Let $g_n\in C^\infty(\cQ_L\times \cQ_L)$ such that $0\leq g_n\leq 1$, $g_n(Q,Q')=0$ if $\rho_L(Q,Q')\leq \frac1{2n}$ and $g_n(Q,Q')=1$ if $\rho_L(Q,Q')\geq \frac1{n}$.
Let $\Psi_t^{n,\eps}$ and $N_t^{n,\eps}$ solve the following SDE
with $\widetilde{Q}^{n,\eps}(t) \eqdef \pi \Psi_t^{n,\eps}$
\begin{equation}\label{eq2.2}
	\aligned
	\dif \Psi_t^{n,\eps} &=\sum_{i=1}^d H_{\Psi_t^{n,\eps}}(\Psi_t^{n,\eps})\circ \dif N_t^{n,\eps},\\
	\dif N_t^{n,\eps}&= \sqrt2(h_{n,\eps}g_n)(Q(t),\widetilde{Q}^{n,\eps}(t))(\Psi_t^{n,\eps})^{-1}
	P_{Q(t),\widetilde{Q}^{n,\eps}(t)}\Phi_t \, \dif B_t
	\\
	&\quad+\sqrt{2(1-(h_{n,\eps}g_n)^2(Q(t),\widetilde{Q}^{n,\eps}(t)))} \, \dif B_t'
	+(\Psi_t^{n,\eps})^{-1}\nabla \cS(\widetilde{Q}^{n,\eps}(t)) \, \dif t,
	\endaligned
\end{equation}
where $B_t'$ is a Brownian motion in $\mR^d$ independent of $B_t$,  $\pi \Psi_0=Q'$
 and $P_{Q,Q'}:T_Q\cQ_L\to T_{Q'}\cQ_L$ is parallel translation along the geodesic from $Q$ to $Q'$.
 As the coefficients are smooth on the compact manifold, we have unique solutions $(\Phi_t,\Psi_t^{n,\eps})$ to \eqref{eq2.1} and \eqref{eq2.2}.

The generator for  $(\Phi_t,\Psi_t^{n,\eps})$ is then given by
\begin{align*}
	\cL^{n,\eps}_{\sO(\cQ_L)}= \Delta_{\sO(\cQ_L),1}+\Delta_{\sO(\cQ_L),2}+2\sum_{i=1}^d(h_{n,\eps}g_n)H_{l_i^*,2}H_{l_i,1}+H_\Phi \nabla \cS+H_\Psi \nabla \cS,
\end{align*}
with $l_i^*(\Phi,\Psi)=\Psi^{-1}P_{\pi\Phi,\pi\Psi}\Phi l_i\in \mR^d$.

We then consider the following approximation to the generator $\cL^c$  defined in \eqref{e:Lc}.
\begin{align*}
	\cL^{n,\eps}=\sum_{e\in E_{\Lambda_{L}}^+}\Delta_{Q_e}+\sum_{e\in E_{\Lambda_{L}}^+}\Delta_{Q_e'}+2g_nh_{n,\eps}\sum_{i,j=1}^{\dim \q_L}\<P_{Q,Q'}v_i,v'_j\>_{T_{Q'}\cQ_L}v_iv'_j+\nabla\cS(Q)+ \nabla\cS(Q'),
\end{align*}
with $\{v_i\}, \{v_j'\}$ as in the definition of $\cL^c$ in  \eqref{e:Lc}.

It is easy to see that $\cL^{n,\eps}_{\sO(\cQ_L)}$ is a lift of $\cL^{n,\eps}$.
Namely, for $f\in C^2(\cQ_L\times \cQ_L)$ and $F(\Phi,\Psi)=f(\pi\Phi,\pi\Psi)$,
 one has
$\cL^{n,\eps}_{\sO(\cQ_L)}F(\Phi,\Psi)=\cL^{n,\eps}f(\pi\Phi,\pi\Psi)$.
We then know $(Q(t),\widetilde{Q}^{n,\eps}_t)=(\pi\Phi_t,\pi\Psi_t)$ starting from $(Q,Q')$ is generated by $\cL^{n,\eps}$
(c.f. \cite[Section~2.1]{Wang}).
Since the marginal operators of $\cL^{n,\eps}$ coincide with $\cL_L$, $(Q_t,\widetilde{Q}^{n,\eps}_t)$ gives a coupling of $\cL_L$-diffusions starting from different initial data.

Let $\mP^Q$ denote the law of $\cL_L$-diffusion $(Q(t))_{t\geq0}$ starting from $Q\in \cQ_L$ in $C([0,\infty);\cQ_L)$ endowed with the distance
\begin{align*}
	\widetilde{\rho}_L(Q,Q')\eqdef\sum_{n=0}^\infty 2^{-n}\Big(1\wedge \sup_{t\in [n,n+1]}\rho_L(Q(t),Q'(t))\Big), \quad Q, Q'\in C([0,\infty);\cQ_L).
\end{align*}
As  the marginal law of  $(Q,\widetilde{Q}^{n,\eps})_{n,\eps}$
is tight in $C([0,\infty);\cQ_L)$, the joint law $\mP^{Q,Q'}_{n,\eps}$ of $(Q,\widetilde{Q}^{n,\eps})_{n,\eps}$  is also tight.
Therefore, for every $\eps>0$ there exists a probability measure
$\mP_\eps^{Q,Q'}$ and a subsequence, which is still denoted by $\mP_{n,\eps}^{Q,Q'}$ such that
$\mP_{n,\eps}^{Q,Q'}\to \mP_\eps^{Q,Q'}$ weakly in $C([0,\infty);\cQ_L)$. Moreover, we could find  $\mP_{\eps_k}^{Q,Q'}$ and $\mP^{Q,Q'}$ such that $\mP_{\eps_k}^{Q,Q'}\to\mP^{Q,Q'}$ weakly in $C([0,\infty);\cQ_L)$.
$\mP^{Q,Q'}$ is then the desired coupling of $\mP^Q$ and $\mP^{Q'}$.


{\bf Proof of inequality \eqref{eq:Ito}.}
In the following we prove \eqref{eq:Ito}.

Following \cite{Ken86}, \cite[Section 2.1]{Wang} we apply It\^o's formula to
$\chi_R(\rho^2)(Q_e(t),\widetilde{Q}^{n,\eps}_e(t))$ and use $\chi'_R\geq0$ to obtain
\begin{equs}[eq:Itoe]
	\dif\chi_R & (\rho^2)(Q_e(t),  \widetilde{Q}^{n,\eps}_e(t))
	\\
	&=\dif M_t^{n,\eps}+2(4\rho^2\chi_R''(\rho^2)+2\chi_R'(\rho^2))(1-g_nh_{n,\eps})(Q(t),\widetilde{Q}^{n,\eps}(t))\dif t-\dif L_t^{n,\eps}
	\\
	&\quad +\1_{C^c\cap D^c}2\chi_R'(\rho^2)\rho\Big( g_nh_{n,\eps}I_{\cS}+(1-g_nh_{n,\eps})Z\Big)(Q(t),\widetilde{Q}^{n,\eps}(t))\dif t.
\end{equs}
Here $\rho=\rho(Q_e(t),\widetilde{Q}^{n,\eps}_e(t))$ and $f(\rho^2)=f(\rho^2)(Q_e(t),\widetilde{Q}^{n,\eps}_e(t))$ for $f\in \{\chi_R', \chi_R''\}$. The term $M_t^{n,\eps}$ is a martingale with quadratic variation process given by
$$
\int_0^t4(2\chi_R'(\rho^2)\rho)^2(1-g_nh_{n,\eps})(Q(s),\widetilde{Q}^{n,\eps}(s))\dif s,
$$
and $L_t^{n,\eps}$ is a non-decreasing process
which increases only when $
(Q(t),\widetilde{Q}^{n,\eps}(t))\in C$. The term $I_{\cS}$ is given in \eqref{e:IS} and finally
\begin{equ}[e:ZQQ]
	Z(Q,Q')=\Delta\rho(\cdot, Q'_e)(Q_e)+\Delta\rho(Q_e, \cdot)(Q_e')+\Big((\nabla \cS)\hat\rho_e(\cdot,Q_e')\Big)(Q)
	+\Big((\nabla \cS)\hat\rho_e(Q_e,\cdot)\Big)(Q').
\end{equ}
with $\Delta$ being the Laplace--Beltrami operator on $G$.

In fact,  to derive \eqref{eq:Itoe},
we may first apply It\^o's formula to $\dif \rho(Q_e(t),\widetilde{Q}^{n,\eps}_e(t))$
and then apply It\^o's formula again to $\chi_R(\rho^2)$.
Suppose for now that the quadratic variation process  of the martingale part $M^\rho$ for $\rho$ is given by
\begin{align}\label{eq:va}
	\<M^\rho_t\>=	4\int_0^t(1-g_nh_{n,\eps})(Q(s),\widetilde{Q}^{n,\eps}(s))\dif s.
\end{align}
Then the second term on the r.h.s. of \eqref{eq:Itoe} comes from $\<M^\rho_t\>$. As explained in \eqref{e:IS}--\eqref{eq:d1} the last line in \eqref{eq:Itoe} comes from $\cL^{n,\eps}\rho$.

We now verify \eqref{eq:va}.
Using \eqref{eq2.1} and \eqref{eq2.2} we have,
for $\widetilde\rho(\Phi,\Psi)=\rho(\pi_e\pi\Phi,\pi_e\pi\Psi)$,
\begin{align*}
	\dif\<M^\rho_t\>
	= 2\sum_{i=1}^d\Big[
	\Big|H_{l_i,1}\widetilde\rho +g_nh_{n,\eps} H_{l_i^*,2}\widetilde\rho \Big|^2
	+\big(1-(g_nh_{n,\eps})^2\big)\big|H_{l_i,2}\widetilde\rho\big|^2
	\Big]\dif t.	
\end{align*}
As $\widetilde\rho$ only depends on the component at $e$,
this now boils down to standard arguments. Namely, since
  the above quantity is independent of the choice of basis,
  we can write the above terms as vector fields on $G$, and
choose basis in such a way that one basis vector is
tangent to the geodesic and the others are perpendicular to the geodesic, and use the first variation formula to derive \eqref{eq:va} (c.f. \cite[Section 6.6]{Hsu}).

In the following we send $n\to\infty, \eps\to 0$ to derive \eqref{eq:Ito}.

Let $(x(t),y(t))$ be the canonical process on $(C([0,\infty);\cQ_L)\times C([0,\infty);\cQ_L),\sF\times \sF)$ and let $\{\sF_t\}_{t\geq0}$ be the natural filtration. On the support of $\chi_R$ using Laplacian comparison theorem (c.f. \cite[Corollary 3.4.4]{Hsu}) we know that $Z$ defined in  \eqref{e:ZQQ}  satisfies $Z\leq C_R$ for some constant $C_R>0$.
Since the support of $\chi_R(\rho^2)\subset \{\rho^2\geq \frac1{2R}\}$ and  for $n$ large enough  $\rho_L(x,y)\geq \rho(x_e,y_e) \geq 1/\sqrt{2R}\geq 1/n$, $g_n(x,y)=1$  on the support of $\chi_R(\rho^2)$, we obtain that
\begin{equation*}
	\aligned
	\chi_R(\rho^2)(Q_e(t),\widetilde{Q}^{n,\eps}_e(t))&-\int_0^t2(4\rho^2\chi_R''(\rho^2)+2\chi_R'(\rho^2))(1-h_{n,\eps})(Q(s),\widetilde{Q}^{n,\eps}(s))
	\\&+2\chi_R'(\rho^2)\rho\Big( h_{n,\eps}J+(1-h_{n,\eps})C_R\Big)(Q(s),\widetilde{Q}^{n,\eps}(s))\,\dif s
	\endaligned
\end{equation*}
is a supermartingale, where $J$ is as in \eqref{eq:Ito}. Therefore,
\begin{align*}
	S_t^{n,\eps}\eqdef	\chi_R(\rho^2)(x_e(t),y_e(t))&-\int_0^t2(4\rho^2\chi_R''(\rho^2)+2\chi_R'(\rho))(x_e(s),y_e(s))(1-h_{n,\eps})(x(s),y(s))
	\\&+2(\chi_R'(\rho^2)\rho)(x_e(s),y_e(s))\Big( h_{n,\eps}J+(1-h_{n,\eps})C_R\Big)(x(s),y(s)) \, \dif s
\end{align*}
is a $\mP_{n,\eps}^{Q,Q'}$-supermartingale.

Furthermore, since by \cite[Lemma~2.1.2]{Wang} $\mP_\eps^{Q,Q'}((x(s),y(s))\in C)=0$ for $s>0$,
using the same argument as in  \cite[Proof part (b) of Theorem 2.1.1]{Wang} we let $n\to\infty$ and obtain that
\begin{align*}
	S_t^\eps\eqdef \chi_R(\rho^2)(x_e(t),y_e(t))&-\int_0^t2\eps(4\rho^2\chi_R''(\rho^2)+2\chi_R'(\rho))(x_e(s),y_e(s))
	\\&+2(\chi_R'(\rho^2)\rho)(x_e(s),y_e(s))\Big( (1-\eps)J+\eps C_R\Big)(x(s),y(s))\,\dif s
\end{align*}
ia a $\mP_{\eps}^{Q,Q'}$-supermartingale.

Letting $\eps\to0$ we obtain that
\begin{align*}
	S_t\eqdef \chi_R(\rho^2)(x_e(t),y_e(t))&-2\int_0^t(\chi_R'(\rho^2)\rho)(x_e(s),y_e(s))J(x(s),y(s)) \, \dif s
\end{align*}
is a $\mP^{Q,Q'}$-supermartingale.
Hence, by Doob--Meyer's decomposition
\begin{align}\label{eq:a1}
	\dif \chi_R\Big(\rho^2(x_e(t),y_e(t))\Big)
	=\dif M_t
	+2(\chi_R'(\rho^2)\rho)
	\Big(x_e(t),y_e(t)\Big)J(x(t),y(t))\dif t
	-\dif L_t,
\end{align}
with $M$ a martingale and $L$ a predictable increasing process.

In the following we prove $M=0$. Similarly we use the above argument for $f(\chi_R(\rho^2(x_e(t),y_e(t))))$ with $0\leq f\in C^2(\mR^+), f'\geq0$ and we have that
\begin{align*}
	f\Big(\chi_R(\rho^2)(x_e(t),y_e(t))\Big)
	-2\int_0^t
	\Big(f'(\chi_R(\rho^2))\chi_R'(\rho^2)\rho\Big)
	\Big(x_e(s),y_e(s)\Big)J(x(s),y(s)) \, \dif s
\end{align*}
is a $\mP^{Q,Q'}$-supermartingale.
Choosing $f(r)=\exp(m r), m\in \mN$ and setting
$$
\Xi_t\eqdef
\exp\Big(m\chi_R\Big(\rho^2(x_e(t),y_e(t))\Big)\Big),
$$ we  have that
\begin{align*}
	\Xi_t&-2m\int_0^t\Xi^1_sJ(x(s),y(s))\,\dif s
\end{align*}
is a $\mP^{Q,Q'}$-supermartingale with $$\Xi_s^1=\Big(\exp(m\chi_R(\rho^2))\chi_R'(\rho^2)\rho\Big)(x_e(s),y_e(s)).$$
By Doob--Meyer's decomposition,
\begin{align}\label{eq:aa1}
	\dif \Xi_t= \dif \widetilde{M}_t+2m\Xi_t^1J(x(t),y(t))\dif t-\dif \widetilde{L}_t,
\end{align}
where $\widetilde{M}_t$ is a martingale and $\widetilde{L}_t$ is predictable increasing.

On the other hand, applying It\^o's formula to \eqref{eq:a1} we obtain
\begin{align}\label{eq:aa2}
	\dif\Xi_t
	=m\,\Xi_t\,\dif M_t
	+\frac12m^2\,\Xi_t \,\dif\<M_t,M_t\>
	+2m\,\Xi_t^1\,J(x(t),y(t))\,\dif t
	-m\,\Xi_t\,\dif L_t.
\end{align}
Comparing \eqref{eq:aa1} and \eqref{eq:aa2} we obtain $\dif\<M,M\>_t\leq \frac2m\dif L_t$. Letting $m\to\infty$, we get $\dif\<M,M\>_t=0$. Hence, \eqref{eq:Ito} follows.

\bibliographystyle{alphaabbr}
\bibliography{refs}

\end{document}